\documentclass{biometrika}

\pdfoutput=1
\RequirePackage{amsmath}

\usepackage{times}
\usepackage{bm}
\usepackage{amsfonts}
\usepackage{amssymb}
\usepackage{graphicx}
\usepackage{esint}
\usepackage{eucal} 
\usepackage{bbm} 
\usepackage[plain,noend]{algorithm2e} 
\usepackage{fixltx2e} 
\usepackage{lscape} 

\usepackage{color}
\usepackage{mathrsfs} 
\usepackage{subfig}

\def\ftil{\widetilde{f}}

  \def\ts{\widetilde{s}}

\def\hbeta{\widehat{\beta}}  
\def\tbeta{\widetilde{\beta}}

\def\hmu{\widehat{\mu}}

\def\tpi{\widetilde{\pi}}

\def\tPi{\widetilde{\Pi}}

\def\htheta{\widehat{\theta}}

\def\ttheta{\widetilde{\theta}}

\begin{document}

\title{Convergence of Regression-Adjusted Approximate Bayesian Computation}

\author{Wentao Li}
\affil{School of Mathematics, Statistics and Physics,  Newcastle University,
Newcastle upon Tyne
NE1 7RU, 
U.K.
\email{wentao.li@newcastle.ac.uk}}
\author{and Paul Fearnhead}
\affil{Department of Mathematics and Statistics, Lancaster University, Lancaster LA1 4YF, U.K.  \email{p.fearnhead@lancaster.ac.uk}}
\maketitle

\begin{abstract}
We present asymptotic results for the regression-adjusted version of approximate Bayesian computation introduced by \cite{Beaumont:2002}. We show that for an appropriate choice
of the bandwidth, regression adjustment will lead to a posterior that, asymptotically, correctly quantifies uncertainty. 
Furthermore, for such a choice of bandwidth we can implement an importance
sampling algorithm to sample from the posterior whose acceptance probability tends to unity as the data sample size increases. This compares favourably to results for standard 
approximate Bayesian computation, where the only way to obtain a posterior
that correctly quantifies uncertainty is to choose a much smaller bandwidth; one for which the acceptance probability tends to zero and hence for which Monte Carlo error will dominate. 
\end{abstract}

Keywords: Approximate Bayesian computation;  Importance sampling; Local-linear regression; Partial information.

\section{Introduction}

Modern statistical applications increasingly require the fitting of complex
statistical models, which are often intractable in the sense that
it is impossible to evaluate the likelihood function. 
This excludes standard implementation of likelihood-based
methods, such as maximum likelihood estimation or Bayesian analysis.
To overcome this there has been substantial interest in
likelihood-free or simulation-based methods, which replace calculating 
the likelihood by simulation of pseudo datasets from the model. Inference
can then be performed by comparing these pseudo datasets, simulated for a range of
different parameter values, to the actual data. 

Examples of such likelihood-free methods include simulated methods of moments \cite[]{duffie1993simulated},
indirect inference \cite[]{Gourieroux:1993,Frigessi:2004}, synthetic likelihood \cite[]{wood2010statistical} and approximate Bayesian computation \cite[]{Beaumont:2002}.  
Of these, approximate Bayesian computation methods are arguably the most common methods
for Bayesian inference, and have been popular in population genetics
\cite[e.g.,][]{Beaumont:2002,Cornuet:2008}, ecology \cite[e.g.,][]{beaumont2010approximate} and
systems biology \cite[e.g.,][]{toni2009approximate}; more recently they have seen 
increased use in other application areas, such as  econometrics \cite[]{calvet2015accurate} and
epidemiology \cite[]{drovandi2011estimation}.

The idea of approximate Bayesian computation is to first summarize the data using low-dimensional summary statistics, such as sample means or autocovariances or suitable quantiles of the data. 
The posterior
density given the summary statistics is then approximated as follows.
Assume the data are $Y_{{\rm obs}}=(y_{{\rm obs},1},\ldots,y_{{\rm obs},n})$
and modelled as a draw from a parametric model 
with parameter $\theta \in \mathbb{R}^p$. Let $K(x)$ be a positive kernel,
where $\max_{x}K(x)=1$, and $\varepsilon>0$ is the bandwidth.
For a given $d$-dimensional summary statistic $s_{}(Y)$, our model will define
a density $f_n(s \mid \theta)$. We then define a joint density, $\pi_{\varepsilon}(\theta,s\mid s_{{\rm obs}})$, for $(\theta,s)$ as
\begin{align} 
 & \frac{\pi(\theta)f_{n}(s\mid\theta)K\{\varepsilon^{-1}(s-s_{{\rm obs}})\}}{\int_{\mathbb{R}^p\times\mathbb{R}^{d}}\pi(\theta)
f_{n}(s\mid\theta)K\{\varepsilon^{-1}(s-s_{{\rm obs}})\}\,d\theta d s},\label{eq:parameter_summary_jointdensity}
\end{align}
where $s_{{\rm obs}}=s_{}(Y_{{\rm obs}})$. Our approximation to the posterior density is the marginal density
\begin{equation} \label{eq:ABC_def}
\pi_{\varepsilon}(\theta\mid s_{{\rm obs}})=\int \pi_{\varepsilon}(\theta,s\mid s_{{\rm obs}})\,d s,
\end{equation}
which we call the approximate Bayesian computation posterior density. For brevity we will often shorten this to posterior density in the following. We will always call the actual posterior given the summary the true posterior. 

The idea of approximate Bayesian computation is that we can sample from $\pi_{\varepsilon}(\theta\mid s_{{\rm obs}})$ without needing to evaluate the likelihood function or $f_{n}(s\mid\theta)$. The 
simplest approach is via rejection sampling \cite[]{Beaumont:2002}, which proceeds by simulating a parameter value and an associated summary statistic from $\pi(\theta)f_n(s\mid\theta)$. This pair
is then accepted with probability $K\{\varepsilon^{-1}(s-s_{{\rm obs}})\}$. The accepted pairs will be drawn from (\ref{eq:parameter_summary_jointdensity}), and the accepted parameter values will be drawn from the 
posterior (\ref{eq:ABC_def}). Implementing this rejection sampler requires only the ability to simulate pseudo data sets from the model, and then to be able to calculate the summary statistics for those data sets.

Alternative algorithms for simulating from the posterior include adaptive or sequential importance sampling 
\cite[]{beaumont2009adaptive,bonassi2015sequential,lenormand2013adaptive,filippi2013optimality} and
Markov chain Monte Carlo approaches \cite[]{Marjoram:2003,Wegmann:2009}. These aim to propose parameter values in areas of high posterior probability, and thus
can be substantially more efficient than rejection sampling. However, the computational efficiency of all these methods is limited by the probability of 
acceptance for data simulated with a parameter value that has high posterior probability. 

This paper is concerned with the asymptotic properties of approximate Bayesian computation. 
It builds upon \cite{Li/Fearnhead:2018} and \cite{Frazier:2016}, who present results on the asymptotic behaviour of the posterior distribution and the posterior mean of approximate Bayesian computation as the amount of data, $n$, increases. 
Their results highlight the tension in approximate Bayesian computation between choices of the summary statistics and bandwidth
that will lead to more accurate inferences, against choices that will reduce the computational cost or Monte Carlo error of algorithms for sampling from the posterior. 

An informal summary of some of these results is as follows. Assume a 
fixed dimensional summary statistic and that the true posterior variance given this summary decreases like $1/n$ as $n$ increases. 
The theoretical results compare the posterior, or  posterior mean, of approximate Bayesian computation, to the true posterior, or true posterior mean, given the summary of the data. 
The accuracy of using approximate Bayesian computation is governed by the choice of bandwidth, and this choice should depend on $n$. \cite{Li/Fearnhead:2018} shows that the optimal
choice of this bandwidth will be $O(n^{-1/2})$. With this choice, estimates based on the posterior mean of approximate Bayesian computation can, asymptotically, be as accurate
as estimates based on the true posterior mean given the summary. Furthermore the Monte Carlo error of an importance sampling algorithm with a good proposal distribution will only inflate the
mean square error of the estimator by a constant factor of the form $1+O(1/N)$, where $N$ is the number of pseudo data sets. These 
results are similar to the asymptotic results for indirect inference, where error for a Monte Carlo sample of size $N$ 
also inflates the overall mean square error of estimators by a factor $1+O(1/N)$ \cite[]{Gourieroux:1993}. By comparison choosing a bandwidth which is $o(n^{-1/2})$ will 
lead to an acceptance probability that tends to
zero as $n\rightarrow \infty$, and the Monte Carlo error of approximate Bayesian computation will blow up. 
Choosing a bandwidth that decays more slowly than $O(n^{-1/2})$ will also lead to a regime where the Monte Carlo error
dominates, and can lead to a non-negligible bias in the posterior mean that inflates the error.

While the above results for a bandwidth that is $O(n^{-1/2})$ are positive in terms of point estimates, they are negative in terms of the calibration of the posterior. 
With such a bandwidth the posterior density of approximate Bayesian computation always over-inflates the parameter uncertainty:
see Proposition \ref{thm:ABC_bad_convergence} below and Theorem 2 of \cite{Frazier:2016}. 

The aim of this paper is to show that a
variant of approximate Bayesian computation can yield inference that is both accurate in terms of point estimation, with its posterior mean having the same frequentist asymptotic variance as the true posterior
mean given the summaries, and calibrated, in the sense that its posterior variance equals
this asymptotic variance, when the bandwidth converges to zero at a rate slower than $O(n^{-1/2})$. This means that the acceptance probability of a good approximate Bayesian computation
algorithm  will tend to unity as $n\rightarrow \infty$.

\section{Notation and Set-up}

We denote the data by $Y_{{\rm obs}}=(y_{{\rm obs},1},\ldots,y_{{\rm obs},n})$,
where $n$ is the sample size, and each observation, $y_{{\rm obs},i}$,
can be of arbitrary dimension. Assume the data are modelled as a draw
from a parametric density, $f_{n}(y\mid\theta)$, and consider asymptotics as $n\rightarrow\infty$. This density depends on
an unknown parameter $\theta\in\mathbb{R}^{p}$. Let $\mathscr{B}^p$ be the Borel sigma-field on $\mathbb{R}^p$. We will let $\theta_{0}$
denote the true parameter value, and $\pi(\theta)$ the prior distribution
for the parameter. Denote the support of $\pi(\theta)$ by $\mathcal{P}$.
Assume that a fixed-dimensional summary statistic $s_n(Y)$ is chosen and
its density under our model is $f_{n}(s\mid\theta)$.
The shorthand $S_n$ is used to denote the random variable with
density $f_{n}(s\mid\theta)$. Often we will simplify notation and write $s$ and $S$ for $s_n$ and $S_n$ respectively. Let $N(x;\mu,\Sigma)$ be the normal density at $x$ with mean $\mu$ and variance $\Sigma$. Let $A^{c}$ be the
complement of a set $A$ with respect to the whole space. For a series
$x_{n}$ 
we write $x_{n}=\Theta(a_n)$
if there exist constants $m$ and $M$ such that $0<m<|x_{n}/a_{n}|<M<\infty$ as $n\rightarrow\infty$. For a real function $g(x)$, denote its gradient function at $x=x_{0}$ by $D_{x}g(x_{0})$. 
To simplify notation, $D_{\theta}$ is written as $D$. Hereafter $\varepsilon$ is considered to depend on $n$, so the notation $\varepsilon_n$ is used. 

The conditions of the theoretical results are stated below. 
\begin{condition} \label{cond:par_true} \label{cond:prior_regular}
There exists some $\delta_{0}>0$, such
that $\mathcal{P}_{0}=\{\theta:|\theta-\theta_{0}|<\delta_{0}\}\subset\mathcal{P}$, $\pi(\theta)\in C^{2}(\mathcal{P}_{0})$ and $\pi(\theta_{0})>0$. 
\end{condition}

\begin{condition} \label{cond:kernel_prop}
The kernel satisfies
(i) $\int vK(v)\,dv=0$; (ii) $\int\prod_{k=1}^{l}v_{i_{k}}K(v)\,dv<\infty$
for any coordinates $(v_{i_{1}},\ldots,v_{i_{l}})$ of $v$ and $l\leq p+6$; (iii) $K(v)\propto\overline{K}(\|v\|_{\Lambda}^{2})$ where $\|v\|_{\Lambda}^{2}=v^{T}\Lambda v$
and $\Lambda$ is a positive-definite matrix, and $K(v)$ is a decreasing
function of $\|v\|_{\Lambda}$; (iv) $K(v)=O(e^{-c_{1}\|v\|^{\alpha_{1}}})$ for some $\alpha_{1}>0$
and $c_{1}>0$ as $\|v\|\rightarrow\infty$.
\end{condition}

\begin{condition} \label{cond:sum_conv}
There exists a sequence $a_{n}$, satisfying
$a_{n}\rightarrow\infty$ as $n\rightarrow\infty$, a $d$-dimensional
vector $s(\theta)$ and a $d\times d$ matrix $A(\theta)$, such
that for all $\theta\in\mathcal{P}_{0}$, 
\[
a_{n}\{S_{n}-s(\theta)\}\rightarrow N\{0,A(\theta)\},\ \ n\rightarrow\infty,
\]
in distribution. We also assume that $s_{\rm obs}\rightarrow s(\theta_0)$ in probability. Furthermore, (i) $s(\theta)$ and $A(\theta)\in C^{1}(\mathcal{P}_{0})$, and
$A(\theta)$ is positive definite for any $\theta$; (ii) for any $\delta>0$ there exists $\delta'>0$ such that $\|s(\theta)-s(\theta_0)\|>\delta'$ for all $\theta$ satisfying $\|\theta-\theta_0\|>\delta$;
and (iii) $I(\theta)=Ds(\theta)^{T}A^{-1}(\theta)Ds(\theta)$ has full rank at $\theta=\theta_{0}$. 
\end{condition}

Let $\widetilde{f}_{n}(s\mid\theta)=N\{s;s(\theta),A(\theta)/a_{n}^{2}\}$ be the density of the normal approximation to $S$
and introduce the standardized random variable $W_{n}(s)=a_{n}A(\theta)^{-1/2}\{S-s(\theta)\}$. We further let
$f_{W_n}(w\mid \theta)$ and $\tilde{f}_{W_n}(w\mid \theta)$ be the densities for $W_n$ under the true model for $S$ and under
our normal approximation to the model for $S$.

\begin{condition} \label{cond:sum_approx}
There exists $\alpha_{n}$ satisfying
$\alpha_{n}/a_{n}^{2/5}\rightarrow\infty$ and a density $r_{\rm max}(w)$
satisfying Condition \ref{cond:kernel_prop} (ii)--(iii) where $K(v)$ is replaced with $r_{\rm max}(w)$, such that $\sup_{\theta\in\mathcal{P}_{0}}\alpha_{n}\left|f_{W_{n}}(w\mid\theta)-\widetilde{f}_{W_{n}}(w\mid\theta)\right|\leq c_{3}r_{\rm max}(w)$
for some positive constant $c_{3}$.
\end{condition}

\begin{condition} \label{cond:sum_approx_tail}
The following statements hold: (i) $r_{\rm max}(w)$ satisfies Condition \ref{cond:kernel_prop} (iv); 
and 
(ii) $\sup_{\theta\in\mathcal{P}_{0}^{c}}f_{W_{n}}(w\mid\theta)=O(e^{-c_{2}\|w\|^{\alpha_{2}}})$
as $\|w\|\rightarrow\infty$ for some positive constants $c_{2}$
and $\alpha_{2}$, and $A(\theta)$ is bounded in $\mathcal{P}$.
\end{condition}

Conditions \ref{cond:par_true}--\ref{cond:sum_approx_tail} are from \cite{Li/Fearnhead:2018}. Condition \ref{cond:kernel_prop} is a requirement for
the kernel function and is satisfied by all commonly used kernels, such as any kernel with compact support or the Gaussian kernel. 
Condition \ref{cond:sum_conv} assumes a central limit theorem for the summary statistic
with rate $a_{n}$, and, roughly speaking, requires the summary statistic to accumulate information when $n$. 
This is a natural assumption, since many common summary statistics are sample moments, proportions, quantiles and autocorrelations, for which
a central limit theorem would apply. 
It is also possible to verify the asymptotic normality of auxiliary model-based or composite likelihood-based summary statistics \cite[]{drovandi2015bayesian,ruli2016approximate}
by referring to the rich literature on asymptotic properties of quasi maximum-likelihood estimators \cite[]{varin2011overview} 
or quasi-posterior estimators \cite[]{chernozhukov2003mcmc}. This assumption does not cover ancillary statistics, using the full data as a summary statistic, or
statistics based on distances, such as an asymptotically chi-square distributed test statistic. Condition \ref{cond:sum_approx} assumes that, 
in a neighborhood of $\theta_{0}$, $f_{n}(s\mid\theta)$ deviates from the leading term of its Edgeworth expansion by a rate $a_{n}^{-2/5}$. 
This is weaker than the standard requirement, $o(a_n^{-1})$, for the remainder from Edgeworth expansion. It also assumes that the deviation is uniform, 
which is not difficult to satisfy in a compact neighborhood. Condition \ref{cond:sum_approx_tail} further assumes that $f_{n}(s\mid\theta)$ has exponentially decreasing tails 
with rate uniform in the support of $\pi(\theta)$. This implies that posterior moments from approximate Bayesian computation are dominated by integrals in the neighborhood of $\theta_0$ and 
have leading terms with concise expressions. With Condition \ref{cond:sum_approx_tail} weakened, the requirement of $\varepsilon_n$ for the proper convergence to 
hold might depend on the specific tail behavior of $f_{n}(s\mid\theta)$. 

Additionally, for the results regarding regression adjustment the following moments of the summary statistic are required to exist.
\begin{condition} \label{cond:likelihood_moments}
The first two moments, $\int_{\mathbb{R}^{d}}s f_{n}(s\mid\theta)\,d s$
and $\int_{\mathbb{R}^{d}}ss^{T}f_{n}(s\mid\theta)\,d s$, exist.
\end{condition}

\section{Asymptotics of approximate Bayesian computation}

\subsection{Posterior}\label{ABC_posterior}

First we consider the convergence of the posterior distribution of approximate Bayesian computation,
denoted by $\Pi_{\varepsilon}(\theta\in A\mid s_{{\rm obs}})$ for $A\in\mathscr{B}^p$,
as $n\rightarrow\infty$. The distribution function is a
random function with the randomness due to $s_{{\rm obs}}$.
We present two convergence results. One is the convergence of the 
posterior distribution function
of a properly scaled and centered version of $\theta$, see Proposition \ref{thm:ABC_bad_convergence}. The other is the convergence
of the posterior mean, a result which comes from \cite{Li/Fearnhead:2018} but, for convenience, is repeated as Proposition \ref{thm:ABC_mean_convergence}. 

The following proposition gives three different limiting forms for $\Pi_{\varepsilon}(\theta\in A\mid s_{{\rm obs}})$, corresponding to different rates for how the bandwidth
decreases relative to the rate of the central limit theorem in Condition \ref{cond:sum_conv}. We summarize these competing rates by 
defining $c_{\varepsilon}=\lim_{n\rightarrow\infty}a_{n}\varepsilon_{n}$. 


\begin{proposition}
\label{thm:ABC_bad_convergence} Assume Conditions \ref{cond:par_true}--\ref{cond:sum_approx_tail}. Let $\theta_{\varepsilon}$ denote the posterior mean of approximate
Bayesian computation. As $n\rightarrow\infty$, if $\varepsilon_{n}=o(a_{n}^{-3/5})$ then the following convergence holds, depending on the value of $c_{\varepsilon}$.
\begin{itemize}
 \item[(i)] If $c_{\varepsilon}=0$ then
 \[
  \sup_{A\in\mathscr{B}^p}\left|\Pi_{\varepsilon}\{a_n(\theta-\theta_{\varepsilon})\in A\mid s_{{\rm obs}}\}-\int_{A}\psi(t)\,dt\right|\rightarrow0, 
 \]
in probability, where $\psi(t)=N\{t;0,I(\theta_{0})^{-1}\}$.
 \item[(ii)] If $c_{\varepsilon}\in (0,\infty)$ then for any $A\in\mathscr{B}^p$,
\[
  \Pi_{\varepsilon}\{a_n(\theta-\theta_{\varepsilon})\in A\mid s_{{\rm obs}}\}\rightarrow\int_{A}\psi(t)\,dt,
\]
in distribution, where
\[
 \psi(t)\propto \int_{\mathbb{R}^{d}}N[t;c_{\varepsilon}\beta_0\{v-E_{G}(v)\},I(\theta_{0})^{-1}]G(v)\,dv,\ \ \beta_{0}=I(\theta_{0})^{-1}Ds(\theta_{0})^{T}A(\theta_{0})^{-1},
\]
and $G(v)$ is a random density of $v$, with mean $E_{G}(v)$, which depends on $c_{\varepsilon}$ and $Z\sim N(0,I_{d})$. 
 \item[(iii)] If $c_{\varepsilon}=\infty$ then
\[
  \sup_{A\in\mathscr{B}^p}\left|\Pi_{\varepsilon}\{\varepsilon_n^{-1}(\theta-\theta_{\varepsilon})\in A\mid s_{{\rm obs}}\}-\int_{A}\psi(t)\,dt\right|\rightarrow 0,
 \]
in probability, where $\psi(t)\propto K\{Ds(\theta_{0})t\}$.
\end{itemize}
\end{proposition}
For a similar result, under different assumptions, see Theorem 2 of \cite{Frazier:2016}. See also \cite{soubeyrand2015weak} for related convergence results for the true posterior given the summaries for 
some specific choices of summary statistics.

The explicit form of $G(v)$ is stated in the Supplementary Material. When we have the same number of summary statistics and parameters, $d=p$, the limiting distribution simplifies to
\[
\psi(t)\propto \int_{\mathbb{R}^{d}}N\{Ds(\theta_0)t;c_{\varepsilon}v,A(\theta_{0})\}K(v)\,dv. 
\]
The more complicated form in Proposition \ref{thm:ABC_bad_convergence} (ii) above arises from the need to project the summary statistics onto the parameter space. The limiting distribution
may depend on the value of the summary statistic, $s_{obs}$, in the space orthogonal to $Ds(\theta_0)^TA(\theta_0)^{-1/2}$. Hence the limit depends on a random quantity, $Z$, which can be interpreted as the noise in $s_{\rm obs}$. 

The main difference between the three convergence results is the form
of the limiting density $\psi(t)$ for the scaled random variable $a_{n,\varepsilon}(\theta-\theta_{\varepsilon})$, where
$a_{n,\varepsilon}=a_{n}\mathbbm{1}_{c_{\varepsilon}<\infty}+\varepsilon_{n}^{-1}\mathbbm{1}_{c_{\varepsilon}=\infty}$. For case (i) the bandwidth is sufficiently small that the  approximation in approximate 
Bayesian computation due to
accepting summaries close to the observed summary is asymptotically negligible. The asymptotic posterior distribution is Gaussian, and equals the limit of the
true posterior for $\theta$ given the summary. For case (iii) the bandwidth is sufficiently big that this approximation dominates and the asymptotic posterior distribution of approximate 
Bayesian computation is determined by the kernel. 
For case (ii) the approximation is of the same order as the uncertainty in $\theta$, which leads to an
asymptotic posterior distribution that is a convolution of a Gaussian distribution and the kernel. Since the limit distributions of cases (i) and (iii) are non-random in the space $L^1(\mathbb{R}^p)$, 
the weak convergence is strengthened to convergence in probability in $L^1(\mathbb{R}^p)$. See the proof in Appendix A. 

\begin{proposition} {\cite[Theorem 3.1 of][]{Li/Fearnhead:2018}}
\label{thm:ABC_mean_convergence}
Assume conditions of Proposition \ref{thm:ABC_bad_convergence}. As $n\rightarrow\infty$, if $\varepsilon_{n}=o(a_{n}^{-3/5})$,
$a_{n}(\theta_{\varepsilon}-\theta_{0})\rightarrow N\{0,I_{ABC}^{-1}(\theta_{0})\}$ in distribution.
If $\varepsilon_{n}=o(a_{n}^{-1})$ or $d=p$ or the covariance matrix of the kernel is proportional to $A(\theta_{0})$ then $I_{ABC}(\theta_{0})=I(\theta_{0})$. For other cases,
$I(\theta_{0})-I_{ABC}(\theta_{0})$ is semi-positive definite.  
\end{proposition}

Proposition \ref{thm:ABC_mean_convergence} helps us to compare the frequentist variability
in the posterior mean of approximate Bayesian computation with the asymptotic posterior distribution given
in Proposition  \ref{thm:ABC_bad_convergence}.
If $\varepsilon_{n}=o(a_{n}^{-1})$ then the posterior
distribution is  asymptotically normal with
variance matrix $a_{n}^{-2}I(\theta_{0})^{-1}$, and the posterior mean is
also asymptotically normal with the same variance matrix. These results are identical
to those we would get for the true posterior and posterior mean given the summary. 

For an $\varepsilon_{n}$ which is the same order as $a_{n}^{-1}$, the uncertainty in approximate Bayesian computation has rate $a_{n}^{-1}$. 
However the limiting posterior distribution, which is a convolution of the true limiting posterior given the summary
with the kernel, will overestimate the uncertainty by a constant factor. 
If $\varepsilon_{n}$ decreases slower than $a_{n}^{-1}$, the posterior contracts at a rate $\varepsilon_{n}$,
and thus will over-estimate the actual uncertainty by a factor that diverges as $n\rightarrow 0$.

In summary, it is much easier to get approximate Bayesian computation to accurately 
estimate the posterior mean. This is possible 
with $\varepsilon_{n}$ as large
as $o(a_{n}^{-3/5})$ if the dimension of the summary statistic equals that of the parameter. However, accurately estimating the 
posterior variance, or getting the posterior to accurately reflect the 
uncertainty in the parameter, is much harder. As commented in Section 1, this
is only possible for values of $\varepsilon_n$ for which the acceptance probability in
a standard algorithm will go to zero as $n$ increases. In this case the Monte Carlo sample size, and hence the computational cost, of
approximate Bayesian computation will have to increase substantially with $n$. 

As one application of our theoretical results, consider observations that are independent and identically distributed from a parametric density $f(\cdot\mid\theta)$. 
One approach to construct the summary statistics is to use the score vector of some tractable approximating auxiliary model evaluated at the maximum auxiliary 
likelihood estimator \cite[]{drovandi2015bayesian}. \cite{ruli2016approximate} constructs an auxiliary model from a composite likelihood, so the auxiliary likelihood for a single observation is 
$\prod_{i\in\mathscr{I}}f(y\in A_i\mid\theta)$ where $\{A_i: i\in\mathscr{I}\}$ is a set of marginal or conditional events for $y$. Denote the auxiliary score vector
for a single observation by $cl_{\theta}(\cdot\mid\theta)$ and the maximum auxiliary likelihood estimator for our data set by $\htheta_{\rm cl}$. Then the summary statistic, $s$, for any pseudo data set 
$\{y_1,\ldots,y_n\}$ is $\sum_{j=1}^n cl_{\theta}(y_j\mid\htheta_{\rm cl})/n$. 

For $y\sim f(\cdot\mid\theta)$, assume the first two moments of $cl_{\theta}(y\mid\theta_0)$ exist and $cl_{\theta}(y\mid\theta)$ is differentiable at $\theta$. 
Let $H(\theta)=E_{\theta}\{\partial cl_{\theta}(y\mid\theta_{0})/\partial\theta\}$ and $J(\theta)=\mbox{var}_{\theta}\{cl_{\theta}(y\mid\theta_{0})\}$. Then if $\htheta_{\rm cl}$ is consistent for $\theta_0$, 
Condition \ref{cond:sum_conv} is satisfied with
\begin{align*}
n^{1/2}[S-E_{\theta}\{cl_{\theta}(Y\mid\theta_0)\}]\rightarrow N\{0,J(\theta)\},\ \ n\rightarrow\infty,
\end{align*}
in distribution, and with $I(\theta_0)=H(\theta_0)^TJ(\theta_0)^{-1}H(\theta_0)$.

Our results show that the posterior mean of approximate Bayesian computation, using $\varepsilon_n=O(n^{-1/2})$, will have asymptotic variance $I(\theta_0)^{-1}/n$. This is identical
to the asymptotic variance of the maximum composite likelihood estimator \cite[]{varin2011overview}. Furthermore, the posterior variance
will overestimate this just by a constant factor. As we show below, using the regression correction of \cite{Beaumont:2002} will correct this overestimation and produce a posterior that correctly quantifies
the uncertainty in our estimates.  

An alternative approach to construct an approximate posterior using composite likelihood is to use the product of the prior and the composite likelihood. In general, this leads to a poorly calibrated
posterior density which substantially underestimates uncertainty \cite[]{ribatet2012bayesian}. Adjustment of the composite likelihood is needed to obtain calibration, but this involves 
estimation of the curvature and the variance of the composite score \cite[]{pauli2011bayesian}. Empirical evidence that approximate Bayesian computation more accurately quantifies uncertainty than alternative 
composite-based posteriors is given in \cite{ruli2016approximate}.

\subsection{Regression Adjusted Approximate Bayesian Computation}\label{adjABC_posterior}


The regression adjustment of \cite{Beaumont:2002} involves post-processing the output of approximate Bayesian computation to 
try to improve the resulting approximation to the true posterior.
Below we will denote a sample from the posterior of approximate Bayesian computation by $\{(\theta_{i},s_{i})\}_{i=1,\ldots,N}$.
Under the regression adjustment, we obtain a new posterior sample by using $\{\theta_{i}-\hbeta_{\varepsilon}(s_{i}-s_{{\rm obs}})\}_{i=1,\ldots,N}$
where $\hbeta_{\varepsilon}$ is the least square estimate of the
coefficient matrix in the linear model 
\begin{align*}
\theta_{i} & =\alpha+\beta(s_{i}-s_{{\rm obs}})+e_{i},\quad i=1,\ldots,N,
\end{align*}
where $e_{i}$ are independent identically distributed errors. 


We can view the adjusted sample as follows. Define a constant, $\alpha_{\varepsilon}$, and a vector $\beta_{\varepsilon}$ as
\[
(\alpha_{\varepsilon},\beta_{\varepsilon})=\underset{\alpha,\beta}{\arg\min}\ E_{\varepsilon}[\|\theta-\alpha-\beta(s-s_{{\rm obs}})\|^{2}\mid s_{{\rm obs}}],
\]
where expectation is with respect the joint posterior distribution of $(\theta,s)$ given by approximate Bayesian computation.  Then the ideal adjusted posterior is
the distribution of $\theta^*=\theta-\beta_{\varepsilon}(s-s_{{\rm obs}})$ where $(\theta,s) \sim \pi_{\varepsilon}(\theta,s)$. The density of $\theta^*$ is
\[
 \pi_{\varepsilon}^*(\theta^{*}\mid s_{{\rm obs}})=\int_{\mathbb{R}^{d}}\pi_{\varepsilon}\{\theta^{*}+\beta_{\varepsilon}(s-s_{{\rm obs}}),s \mid s_{{\rm obs}}\}\,ds
\]
and the sample we get from regression-adjusted approximate Bayesian computation is a draw from $\pi_{\varepsilon}^*(\theta^{*}\mid s_{{\rm obs}})$ but with $\beta_{\varepsilon}$ replaced by
its estimator.

The variance of $\pi^*_{\varepsilon}(\theta^{*}\mid s_{{\rm obs}})$ is strictly smaller than that of $\pi_{\varepsilon}(\theta\mid s_{{\rm obs}})$
provided $s$ is correlated with $\theta$. The following results, which are analogous
to Propositions \ref{thm:ABC_bad_convergence} and \ref{thm:ABC_mean_convergence}, show 
that this reduction in variation is by the correct amount to make the resulting adjusted posterior correctly quantify the posterior uncertainty.

\begin{theorem}
\label{thm:ABC_good_convergence}Assume Conditions \ref{cond:par_true}--\ref{cond:likelihood_moments}.
Denote the mean of $\pi^*_{\varepsilon}(\theta^{*}\mid s_{{\rm obs}})$ by
$\theta_{\varepsilon}^{*}$. As $n\rightarrow\infty$, if $\varepsilon_{n}=o(a_{n}^{-3/5})$,
\begin{align*}
 & \sup_{A\in\mathscr{B}^p}\left|\Pi_{\varepsilon}\{a_{n}(\theta^{*}-\theta_{\varepsilon}^{*})\in A\mid s_{{\rm obs}}\}-\int_{A}N\{t;0,I(\theta_{0})^{-1}\}\,dt\right|\rightarrow 0,
\end{align*}
in probability, and $a_{n}(\text{\ensuremath{\theta_{\varepsilon}^{*}}}-\theta_{0})\rightarrow N\{0,I(\theta_{0})^{-1}\}$ in distribution. Moreover, if $\beta_{\varepsilon}$ is replaced by $\tbeta_{\varepsilon}$ satisfying $a_n\varepsilon_n(\tbeta_{\varepsilon}-\beta_{\varepsilon})=o_p(1)$, the above results still hold. 
\end{theorem}
The limit of the regression adjusted posterior distribution is the true posterior given the summary provided
$\varepsilon_{n}$ is $o(a_{n}^{-3/5})$. This is a slower rate than that at which the posterior contracts, which, as we will show in the next section, has important consequences
in terms of the computational efficiency of approximate Bayesian computation.
The regression adjustment corrects both the additional noise
of the  posterior mean when $d>p$ and the overestimated uncertainty
of the  posterior. 
This correction comes from the removal
of the first order bias caused by $\varepsilon$. 
\cite{blum2010approximate} shows that the regression adjustment reduces the
bias of approximate Bayesian computation when $E(\theta\mid s)$ is linear
and the residuals $\theta-E(\theta\mid s)$ are homoscedastic. 
Our results do not require these assumptions, and suggest
that the regression adjustment should be applied routinely with approximate Bayesian computation provided the coefficients $\beta_{\varepsilon}$ can be
estimated accurately.

With the simulated sample, $\beta_{\varepsilon}$ is estimated by $\hbeta_{\varepsilon}$. The accuracy of $\hbeta_{\varepsilon}$
can be seen by the following decomposition,
\begin{align*}
\hbeta_{\varepsilon} & =\mbox{cov}_N(s,\theta)\mbox{var}_N(s)^{-1}\\
 & =\beta_{\varepsilon}+\frac{1}{a_{n}\varepsilon_{n}}\mbox{cov}_N\left\{\frac{s-s_{\varepsilon}}{\varepsilon_{n}},a_{n}(\theta^{*}-\theta_{\varepsilon}^{*})\right\}
 \mbox{var}_N\left(\frac{s-s_{\varepsilon}}{\varepsilon_{n}}\right)^{-1},
\end{align*}
where $\mbox{cov}_N$ and $\mbox{var}_N$ are the sample covariance
and variance matrices, and $s_\epsilon$ is the sample mean. Since $\mbox{cov}(s,\theta^{*})=0$ and the distributions of $s-s_{\varepsilon}$ and $\theta^{*}-\theta_{\varepsilon}^{*}$ contract at 
rates $\varepsilon_n$ and $a_n^{-1}$ respectively, the error $\hbeta_{\varepsilon}-\beta_{\varepsilon}$ can be shown to have the rate $O_p\{(a_{n}\varepsilon_{n})^{-1}N^{-1/2}\}$ as $n\rightarrow\infty$ 
and $N\rightarrow\infty$. We omit the proof, since it is tedious and similar to the proof of the asymptotic expansion of $\beta_{\varepsilon}$ in Lemma \ref{lem:reg_coef}. 
Thus, if $N$ increases to infinity with $n$, $\hbeta_{\varepsilon}-\beta_{\varepsilon}$ will be $o_p\{(a_n\varepsilon_n)^{-1}\}$ and the convergence of Theorem \ref{thm:ABC_good_convergence} will hold instead. 

Alternatively we can get an idea of the additional error for large $N$ from the following proposition.
\begin{proposition}\label{prop:ABC_good_convergence}
Assume Conditions \ref{cond:par_true}--\ref{cond:likelihood_moments}. Consider $\theta^{*}=\theta-\hbeta_{\varepsilon}(s-s_{{\rm obs}})$. As $n\rightarrow\infty$, if $\varepsilon_{n}=o(a_{n}^{-3/5})$ and $N$ is large enough, for any $A\in\mathscr{B}^{p}$,
\begin{align*}
\Pi_{\varepsilon}\{a_{n}(\theta^{*}-\theta_{\varepsilon}^{*})\in A\mid s_{{\rm obs}}\}\rightarrow\int_{A}\psi(t)\,dt,
\end{align*}
in distribution, where 
\begin{align*}
\psi(t)\propto & \int_{\mathbb{R}^{d}}N\left[t;\frac{\eta}{N^{1/2}}\{v-E_{G}(v)\},I(\theta_{0})^{-1}\right]G(v)\,dv,
\end{align*}
when $c_{\varepsilon}<\infty$,  
\begin{align*}
\psi(t)\propto & \int_{\mathbb{R}^{p}}N\left\{ t;\frac{\eta}{N^{1/2}}Ds(\theta_{0})t',I(\theta_{0})^{-1}\right\} K\{Ds(\theta_{0})t'\}\,dt',
\end{align*}
when $c_{\varepsilon}=\infty$, and $\eta=O_p(1)$.
\end{proposition}
The limiting distribution here can be viewed as the convolution of the limiting distribution obtained when the optimal coefficients are used and that of a random variable, which
relates to the error in our estimate of $\beta_{\varepsilon}$, and that is $O_p(N^{-1/2})$. 

\subsection{Acceptance Rates when $\varepsilon$ is Negligible}\label{ABCacc_rate}

Finally we present results for the acceptance probability of approximate Bayesian computation, the quantity that
is central to the computational cost of importance sampling or Markov chain Monte Carlo-based algorithms.
We consider a set-up where we 
propose the parameter value from a location-scale family. That is, we can write
the proposal density as the density of a random variable, $\sigma_{n}X+\mu_{n}$,
where $X\sim q(\cdot)$, $E(X)=0$ and $\sigma_n$ and $\mu_n$ are constants that can depend on $n$. The average acceptance
probability $p_{{\rm acc},q}$ would then be 
\begin{align*}
& \int_{\mathcal{P}\times\mathbb{R}^{d}}q_{n}(\theta)f_{n}(s\mid\theta)K\{\varepsilon_{n}^{-1}(s-s_{{\rm obs}})\}\,d s d\theta,
\end{align*}
where $q_{n}(\theta)$ is the density of $\sigma_{n}X+\mu_{n}$. This covers the proposal distribution in fundamental sampling algorithms, including the random-walk Metropolis algorithm 
and importance sampling with unimodal proposal distribution, and serves as the building block for many advanced algorithms where the proposal distribution is a mixture of distributions from 
location-scale families, such as iterative importance sampling-type algorithms. 

We further assume that $\sigma_{n}(\mu_{n}-\theta_{0})=O_{p}(1)$, which means $\theta_{0}$ is in the coverage of $q_{n}(\theta)$. This is a
natural requirement for any good proposal distribution.
The prior distribution and $\theta_{0}$ as a point mass are included
in this proposal family. This condition would also apply to many Markov chain Monte Carlo implementations
of approximate Bayesian computation after convergence.

As above, define $a_{n,\varepsilon}=a_{n}\mathbbm{1}_{c_{\varepsilon}<\infty}+\varepsilon_{n}^{-1}\mathbbm{1}_{c_{\varepsilon}=\infty}$ to be the smaller of $a_{n}$ and $\varepsilon_{n}^{-1}$.
Asymptotic results for $p_{{\rm acc},q}$
when $\sigma_{n}$ has the same rate as $a_{n,\varepsilon}^{-1}$ are given in \cite{Li/Fearnhead:2018}. Here we extend those results to other regimes.
\begin{theorem}
\label{thm:acceptance_rate}Assume the conditions of Proposition \ref{thm:ABC_bad_convergence}. As $n\rightarrow\infty$, if $\varepsilon_{n}=o(a_{n}^{-1/2})$:
(i) if $c_{\varepsilon}=0$ or $\sigma_{n}/a_{n,\varepsilon}^{-1}\rightarrow\infty$, then $p_{{\rm acc},q}\rightarrow0$ in probability; 
(ii) if $c_{\varepsilon}\in(0,\infty)$ and $\sigma_{n}/a_{n,\varepsilon}^{-1}\rightarrow r_{1}\in[0,\infty)$,
or $c_{\varepsilon}=\infty$ and $\sigma_{n}/a_{n,\varepsilon}^{-1}\rightarrow r_{1}\in(0,\infty)$,
then $p_{{\rm acc},q}=\Theta_{p}(1)$; 
 (iii) if $c_{\varepsilon}=\infty$ and $\sigma_{n}/a_{n,\varepsilon}^{-1}\rightarrow0$,
then $p_{{\rm acc},q}\rightarrow1$ in probability.
\end{theorem}
The proof of Theorem \ref{thm:acceptance_rate} can be found in the Supplementary Material. The underlying intuition is as follows. For the summary
statistic, $s_{}$, sampled with parameter value $\theta$, the
acceptance probability depends on 
\begin{align}
\frac{s_{}-s_{{\rm obs}}}{\varepsilon_{n}} & =\frac{1}{\varepsilon_{n}}[\{s_{}-s(\theta)\}+\{s(\theta)-s(\theta_{0})\}+\{s(\theta_{0})-s_{{\rm obs}}\}],\label{eq:acceptance_rate_explain}
\end{align}
where $s(\theta)$ is the limit of $s_{}$ in Condition \ref{cond:sum_conv}.
The distance between $s_{}$ and $s_{{\rm obs}}$ is at least $O_{p}(a_{n}^{-1})$,
since the first and third bracketed terms are $O_{p}(a_{n}^{-1})$. If $\varepsilon_{n}=o(a_{n}^{-1})$ then,
regardless of the value of $\theta$, \eqref{eq:acceptance_rate_explain}
will blow up as $n\rightarrow\infty$ and hence $p_{{\rm acc},q}$ goes
to $0$. If $\varepsilon_{n}$ decreases with a rate slower than $a_{n}^{-1}$, \eqref{eq:acceptance_rate_explain}
will go to zero providing we have a proposal which ensures that the middle term is $o_{p}(\varepsilon_{n})$, and hence $p_{{\rm acc},q}$ goes
to unity.

Theorem \ref{thm:ABC_bad_convergence} shows that, without the regression adjustment,
approximate Bayesian computation requires $\varepsilon_{n}$ to be $o(a_{n}^{-1})$ if its posterior is to converge to the 
true posterior given the summary. In this case Theorem \ref{thm:acceptance_rate} shows that the acceptance
rate will degenerate to zero as $n\rightarrow\infty$ regardless of the
choice of $q(\cdot)$. 
On the other hand, with the regression adjustment, we can choose
$\varepsilon_{n}=o(a_{n}^{-3/5})$ and still have convergence to the true posterior given the summary. 
For such a choice, if our proposal density satisfies $\sigma_n=o(\varepsilon_{n})$, the acceptance rate will go to unity as $n\rightarrow\infty$. 

\section{Numerical Example}
Here we illustrate the gain of computational efficiency from using the regression adjustment on the $g$-and-$k$ distribution, 
a popular model for testing approximate Bayesian computation methods 
\cite[e.g.,][]{fearnhead2012constructing,marin2014relevant}. 
The data are independent and identically distributed from a distribution defined by its quantile function, 
\begin{align*}
F^{-1}(x;\alpha,\beta,\gamma,\kappa)= & \alpha+\beta\left[1+0.8\frac{1-\exp\{-\gamma z(x)\}}{1+\exp\{-\gamma z(x)\}}\right]\{1+z(x)^{2}\}^{\kappa}z(x),\ \ x\in[0,1],
\end{align*}
where $\alpha$ and $\beta$ are location and scale parameters, $\gamma$ and
$\kappa$ are related to the skewness and kurtosis of the distribution, and $z(x)$ is the corresponding quantile of a standard normal distribution.
No closed form is available for the density but simulating from the
model is straightforward by transforming realisations from the standard normal distribution. 

In the following we assume the parameter vector $(\alpha,\beta,\gamma,\kappa)$ has a uniform prior in $[0,10]^{4}$ and multiple datasets are generated from the model with $(\alpha,\beta,\gamma,\kappa)=(3,1,2,0.5)$.
To illustrate the asymptotic behaviour of approximate Bayesian computation, $50$ data sets are generated for each of a set of values of $n$ ranging from $500$ to $10,000$. 
Consider estimating the posterior means, denoted by $\mu=(\mu_{1},\ldots,\mu_{4})$,
and standard deviations, denoted by $\sigma=(\sigma_{1},\ldots,\sigma_{4})$, of the parameters. 
The summary statistic is a set of evenly spaced quantiles of dimension $19$. 

The bandwidth is chosen via fixing the proportion of the Monte Carlo sample to be accepted, and the accepted proportions needed to achieve certain approximation accuracy for estimates with and without the adjustment are compared. 
A higher proportion means more simulated parameter values can be kept for inference. The accuracy is measured by the average relative
errors of estimating $\mu$ or $\sigma$, 
\begin{align*}
{\rm RE}_{\mu}=\frac{1}{4}\sum_{k=1}^{4}\frac{\left|\widehat{\mu}_{k}-\mu_{k}\right|}{\mu_{k}},\quad {\rm RE}_{\sigma}=\frac{1}{4}\sum_{k=1}^{4}\frac{\left|\widehat{\sigma}_{k}-\sigma_{k}\right|}{\sigma_{k}},
\end{align*}
for estimators $\hmu=(\hmu_1,\ldots,\hmu_4)$ and $\widehat{\sigma}=(\widehat{\sigma}_1,\ldots,\widehat{\sigma}_4)$. The proposal distribution is normal, with the covariance matrix selected to inflate the posterior covariance matrix by a constant factor $c^{2}$ and the mean vector selected to differ from the posterior mean by half of the
posterior standard deviation, which avoids the case that the posterior mean can be estimated trivially. We consider a series of increasing $c$ in order to investigate the impact of the proposal distribution getting worse. 

Figure \ref{gk_efficiency} shows that the required acceptance rate for the regression adjusted estimates is higher than that for the unadjusted estimates in almost all cases. 
For estimating the posterior mean, the improvement is small. For estimating the posterior standard deviations, 
the improvement is much larger. To achieve each level of accuracy, the acceptance rates of the unadjusted estimates are all close to zero. 
Those of the regression-adjusted estimates are higher by up to two orders of magnitude,
so the Monte Carlo sample size needed to achieve the same accuracy can be reduced correspondingly. 

\begin{figure}
\centering
\includegraphics[scale=0.8]{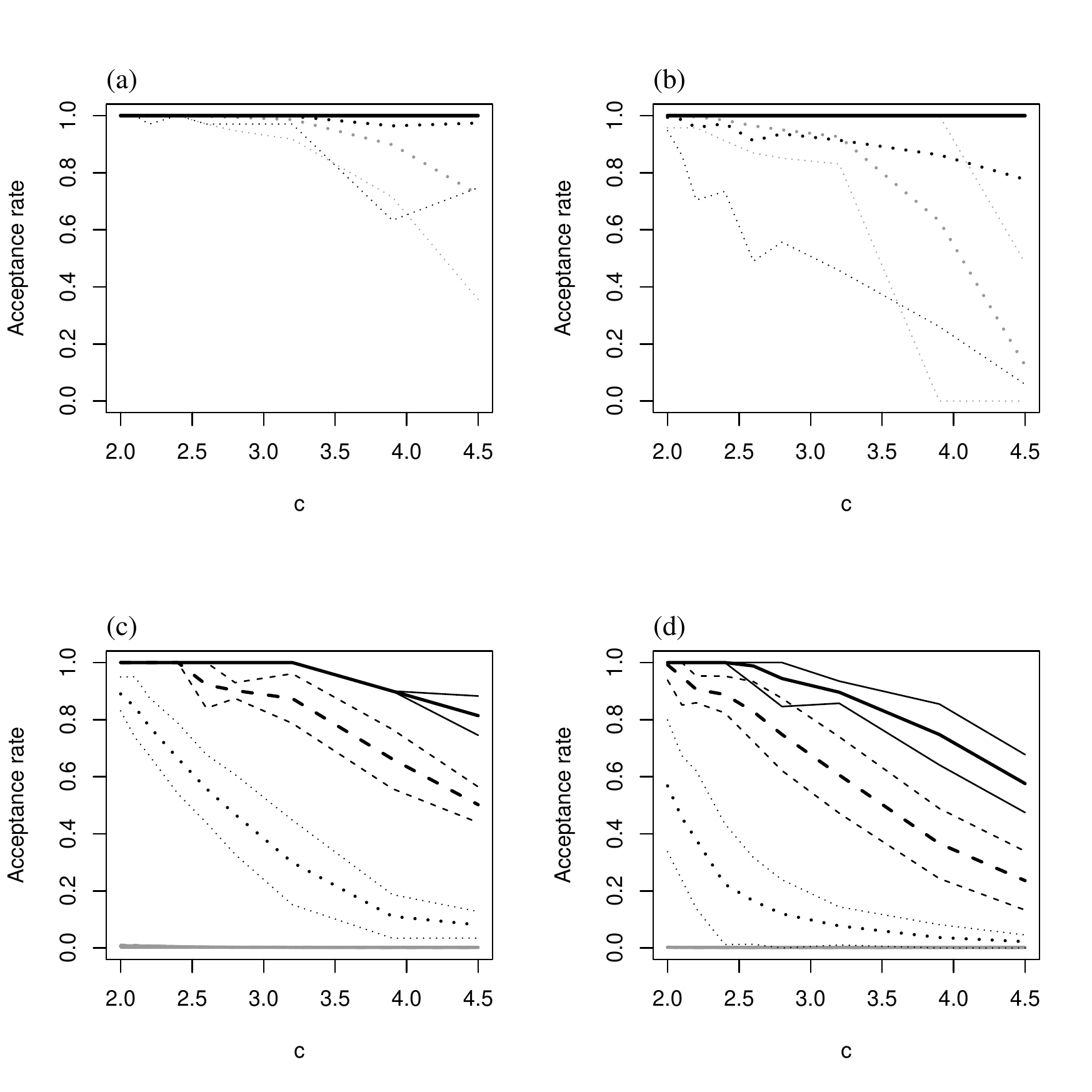}
\caption{
\label{gk_efficiency} Acceptance rates required for different degrees of accuracy of approximate Bayesian computation and different variances of the proposal distribution (which are proportional to $c$). 
In each plot we show results for standard (grey-line) and regression adjusted (black-line) approximate Bayesian computation and for
different values of $n$: $n=500$ (dotted), $n=3,000$ (dashed) and $n=10,000$ (solid). The averages over $50$ data sets (thick) and their $95\%$ confidence intervals (thin) are reported. Results are for a relative error of 0.08 and 0.05 in the posterior mean, in (a) and (b) respectively, and
for a relative error of 0.2 and 0.1 in the posterior standard deviation, in (c) and (d) respectively.}
\end{figure}

\section{Discussion}

One way to implement approximate Bayesian computation so that the acceptance probability tends to unity as $n$ increases is to use importance sampling with a suitable proposal from a location-scale family. 
The key difficulty with finding a suitable proposal is to ensure that the location parameter 
is close to the true parameter, where close means the distance is $O(\varepsilon_n)$. This can be achieved by having a preliminary analysis of the data, and using the point estimate of the parameter 
from this preliminary analysis as the location parameter \cite[]{beaumont2009adaptive,Li/Fearnhead:2018}.


\section*{Acknowledgment}

This work was funded by the Engineering and Physical Sciences Research Council, under the i-like programme grant.

\section*{Supplementary Material}

Proofs of lemmas and Theorem \ref{thm:acceptance_rate} are included in the online supplementary material.


\section*{Appendix}
\addcontentsline{toc}{section}{Appendix}

\subsection*{Proof of Result from Section \ref{ABC_posterior}}

Throughout the data are considered to be random.  For any integer $l>0$ and a set $A\subset\mathbb{R}^{l}$,
we use the convention that $cA+x$ denotes the set $\{ct+x:t\in A\}$
for $c\in\mathbb{R}$ and $x\in\mathbb{R}^{l}$. For a non-negative
function $h(x)$, integrable in $\mathbb{R}^{l}$, denote the normalised
function $h(x)/\int_{\mathbb{R}^{l}}h(x)\,dx$ by $h(x)^{({\rm norm})}$.
For a vector $x$, denote a general polynomial of elements of $x$ with
degree up to $l$ by $P_{l}(x)$. For any fixed $\delta<\delta_{0}$,
let $B_{\delta}$ be the neighborhood $\{\theta:\|\theta-\theta_{0}\|<\delta\}$.
Let $\pi_{\delta}(\theta)$ be $\pi(\theta)$ truncated in $B_{\delta}$.
Recall that $\ftil_{n}(s\mid\theta)$ denotes the normal density with
mean $s(\theta)$ and covariance matrix $A(\theta)/a_{n}^{2}$. Let
$t(\theta)=a_{n,\varepsilon}(\theta-\theta_{0})$ and $v(s)=\varepsilon_{n}^{-1}(s-s_{{\rm obs}})$,
rescaled versions $\theta$ and $s$. For any $A\in\mathscr{B}^{p}$,
let $t(A)$ be the set $\{\phi:\phi=t(\theta)\text{ for some }\theta\in A\}$. 

Define $\tPi_{\varepsilon}(\theta\in A\mid s_{{\rm obs}})$ to be the
normal counterpart of $\Pi_{\varepsilon}(\theta\in A\mid s_{{\rm obs}})$
with truncated prior, obtained by replacing $\pi(\theta)$ and $f_{n}(s\mid\theta)$
in $\Pi_{\varepsilon}$ by $\pi_{\delta}(\theta)$ and $\widetilde{f}_{n}(s\mid\theta)$.
So let 
\[
\tpi_{\varepsilon}(\theta,s\mid s_{{\rm obs}})=\frac{\pi_{\delta}(\theta)\ftil_{n}(s\mid \theta)K\{\varepsilon_{n}^{-1}(s-s_{{\rm obs}})\}}{\int_{B_{\delta}}\int_{\mathbb{R}^{d}}\pi_{\delta}(\theta)\ftil_{n}(s\mid \theta)K\{\varepsilon_{n}^{-1}(s-s_{{\rm obs}})\}\,d\theta ds},
\]
$\tpi_{\varepsilon}(\theta\mid s_{{\rm obs}})=\int_{\mathbb{R}^{d}}\tpi_{\varepsilon}(\theta,s\mid s_{{\rm obs}})\,ds$
and $\tPi_{\varepsilon}(\theta\in A\mid s_{{\rm obs}})$ be the distribution
function with density $\tpi_{\varepsilon}(\theta\mid s_{{\rm obs}})$.
Denote the mean of $\tPi_{\varepsilon}$ by $\ttheta_{\varepsilon}$. Let $W_{{\rm obs}}=a_{n}A(\theta_{0})^{-1/2}\{s_{{\rm obs}}-s(\theta_{0})\}$
and $\beta_{0}=I(\theta_{0})^{-1}Ds(\theta_{0})^{T}A(\theta_{0})^{-1}$.
By Condition \ref{cond:sum_conv}, $W_{{\rm obs}}\rightarrow Z$ in
distribution as $n\rightarrow\infty$, where $Z\sim N(0,I_{d})$.

Since the approximate Bayesian computation likelihood within $\tPi_{\varepsilon}$ is an incorrect
model for $s_{{\rm obs}}$, standard posterior convergence results
do not apply. However, if we condition on the value of the summary, $s$, then the 
distribution of $\theta$ is just the true posterior given $s$. Thus we can express the posterior from approximate Bayesian computation as a continuous mixture of these true posteriors.
Let $\tpi_{\varepsilon,tv}(t,v)=a_{n,\varepsilon}^{-d}\pi_{\delta}(\theta_{0}+a_{n,\varepsilon}^{-1}t)\ftil_{n}(s_{{\rm obs}}+\varepsilon_{n}v \mid \theta_{0}+a_{n,\varepsilon}^{-1}t)K(v)$.
For any $A\in\mathscr{B}^{p}$, we rewrite $\tPi_{\varepsilon}$ as, 
\begin{equation}
\tPi_{\varepsilon}(\theta\in A\mid s_{{\rm obs}})=\int_{\mathbb{R}^{d}}\int_{t(B_{\delta})}\tPi(\theta\in A\mid s_{{\rm obs}}+\varepsilon_{n}v)\tpi_{\varepsilon,tv}(t,v)^{({\rm norm})}\,dtdv,\label{eq:ABC_posterior_altform1}
\end{equation}
where $\tPi(\theta\in A\mid s)$ is the posterior distribution with prior
$\pi_{\delta}(\theta)$ and likelihood $\widetilde{f}_{n}(s\mid\theta)$.

Using results from \cite{kleijn2012bernstein}, the leading term of $\tPi(\theta\in A\mid s_{{\rm obs}}+\varepsilon_{n}v)$
can be obtained and is stated in the following lemma.
\begin{lemma}\label{lem:misspecified_posterior_limit}Assume Conditions \ref{cond:sum_conv} and \ref{cond:sum_approx}. If $\varepsilon_{n}=O(a_{n}^{-1})$, for any fixed $v\in\mathbb{R}^{d}$ and small enough $\delta$,
\[
\sup_{A\in\mathscr{B}^{p}}\left|\tPi\{a_{n}(\theta-\theta_{0})\in A\mid s_{{\rm obs}}+\varepsilon_{n}v\}-\int_{A}N[t;\beta_{0}\{A(\theta_{0})^{1/2}W_{{\rm obs}}+c_{\varepsilon}v\},I(\theta_{0})^{-1}]\,dt\right|\rightarrow0,
\]
in probability as $n\rightarrow\infty$.
\end{lemma}
This leading term is Gaussian, but with a mean that depends on $v$. Thus asymptotically, the posterior of approximate Bayesian computation is the distribution of the 
sum of a Gaussian random variable and $\beta_0c_{\varepsilon} V$, where $V$ has density proportional to $\int \tpi_{\varepsilon,tv}(t,v)\,dt.$

To make this argument rigorous, and to find the distribution of this sum of random variables we need to introduce several functions that relate to the limit of
$\tpi_{\varepsilon,tv}(t',v)$.
For a rank-$p$ $d\times p$ matrix $A$, a rank-$d$ $d\times d$
matrix $B$ and a $d$-dimensional vector $c$, define $g(v;A,B,c)=\exp[-(c+Bv)^{T}\{I-A(A^{T}A)^{-1}A^{T}\}(c+Bv)/2]/(2\pi)^{(d-p)/2}$. Let 
\[
g_{n}(t,v)=\begin{cases}
N\left\{ Ds(\theta_{0})t;a_{n}\varepsilon_{n}v+A(\theta_{0})^{1/2}W_{{\rm obs}},A(\theta_{0})\right\} K(v), & \ c_{\varepsilon}<\infty,\\
N\left\{ Ds(\theta_{0})t;v+\frac{1}{a_{n}\varepsilon_{n}}A(\theta_{0})^{1/2}W_{{\rm obs}},\frac{1}{a_{n}^{2}\varepsilon_{n}^{2}}A(\theta_{0})\right\} K(v), & \ c_{\varepsilon}=\infty,
\end{cases}
\]
$G_{n}(v)$ be $g\{v;A(\theta_{0})^{-1/2}Ds(\theta_{0}),a_{n}\varepsilon_{n}A(\theta_{0})^{-1/2},W_{{\rm obs}}\}K(v)$,
and $E_{G_{n}}(\cdot)$ be the expectation under the density $G_{n}(v)^{({\rm norm})}$.
In both cases it is straightforward to show that $\int_{\mathbb{R}^{p}}g_{n}(t,v)\,dt=|A(\theta_{0})|^{-1/2}G_{n}(v)$.
Additionally, for the case $c_{\varepsilon}=\infty$, define $v'(v,t)=A(\theta_{0})^{1/2}W_{{\rm obs}}+a_{n}\varepsilon_{n}v-a_{n}\varepsilon_{n}Ds(\theta_{0})t$
and 
\begin{align*}
g'_{n}(t,v')=N\{v';0,A(\theta_{0})\}K\left\{ Ds(\theta_{0})t+\frac{1}{a_{n}\varepsilon_{n}}v'-\frac{1}{a_{n}\varepsilon_{n}}A(\theta_{0})^{1/2}W_{{\rm obs}}\right\} .
\end{align*}
Then with the transformation $v'=v'(v,t)$, $g_{n}(t,v)dv=g_{n}'(t,v')dv'$.
Let 
\begin{align*}
g(t,v) & =\begin{cases}
N\{Ds(\theta_{0})t;c_{\varepsilon}v+A(\theta_{0})^{1/2}Z,A(\theta_{0})\}K(v), & \ c_{\varepsilon}<\infty,\\
K\{Ds(\theta_{0})t\}N\{v;0,A(\theta_{0})\}, & \ c_{\varepsilon}=\infty,
\end{cases}
\end{align*}
$G(v)$ be $g\{v;A(\theta_{0})^{-1/2}Ds(\theta_{0}),c_{\varepsilon}A(\theta_{0})^{-1/2},Z\}K(v)$
and $E_{G}(\cdot)$ be the expectation under the density $G(v)^{({\rm norm})}$.
When $c_{\varepsilon}<\infty$, $\int_{\mathbb{R}^{p}}g(t,v)\,dt=|A(\theta_{0})|^{-1/2}G(v)$. 

Expansions of $\int_{\mathbb{R}^{d}}\int_{t(B_{\delta})}\tpi_{\varepsilon,tv}(t,v)dtdv$
are given in the following lemma. 

\begin{lemma}\label{lem:ABC_likelihood_expansion} Assume Conditions \ref{cond:prior_regular}--\ref{cond:sum_conv}. If $\varepsilon_{n}=o(a_{n}^{-1/2})$, then $\int_{\mathbb{R}^{d}}\int_{t(B_{\delta})}\left|\tpi_{\varepsilon,tv}(t,v)-\pi(\theta_{0})g_{n}(t,v)\right|\,dtdv\rightarrow0$
in probability and $\int_{\mathbb{R}^{d}}\int_{t(B_{\delta})}g_{n}(t,v)\,dtdv=\Theta_{p}(1)$, as $n\rightarrow\infty$.
Furthermore, for $l\leq6$, $\int_{\mathbb{R}^{d}}\int_{t(B_{\delta})}P_{l}(v)g_{n}(t,v)\,dtdv$
converges to $|A(\theta_{0})|^{-1/2}\int_{\mathbb{R}^{d}}P_{l}(v)G(v)\,dv$ in distribution
when $c_{\varepsilon}<\infty$ and converges to $\int_{\mathbb{R}^{p}}P_{l}\{Ds(\theta_{0})t\}K\{Ds(\theta_{0})t\}\,dt$
in probability when $c_{\varepsilon}=\infty$, as $n\rightarrow\infty$.
\end{lemma}

The following lemma states that $\Pi_{\varepsilon}$ and $\tPi_{\varepsilon}$
are asymptotically the same and gives an expansion of $\ttheta_{\varepsilon}$. 

\begin{lemma} \label{lem:ABC_probability_expand}Assume Conditions
\ref{cond:par_true}--\ref{cond:sum_approx}.
If $\varepsilon_{n}=o(a_{n}^{-1/2})$, then

(a) for any $\delta<\delta_{0}$, $\Pi_{\varepsilon}(\theta\in B_{\delta}^{c}\mid s_{{\rm obs}})$
and $\tPi_{\varepsilon}(\theta\in B_{\delta}^{c}\mid s_{{\rm obs}})$
are $o_{p}(1)$;

(b) there exists a $\delta<\delta_0$ such that $\sup_{A\in\mathscr{B}^{p}}\left|\Pi_{\varepsilon}(\theta\in A\cap B_{\delta}\mid s_{{\rm obs}})-\tPi_{\varepsilon}(\theta\in A\cap B_{\delta}\mid s_{{\rm obs}})\right|=o_{p}(1)$;

(c) if, in addition, Condition \ref{cond:sum_approx_tail} holds, then $a_{n,\varepsilon}(\theta_{\varepsilon}-\ttheta_{\varepsilon})=o_{p}(1)$,
and $\ttheta_{\varepsilon}=\theta_{0}+a_{n}^{-1}\beta_{0}A(\theta_{0})^{1/2}W_{{\rm obs}}+\varepsilon_{n}\beta_{0}E_{G_{n}}(v)+r_{n,1}$
where the remainder $r_{n,1}=o_{p}(a_{n}^{-1})$. 

\end{lemma}

\begin{proof}[of Proposition \ref{thm:ABC_bad_convergence}] Lemma 10 in the
Supplementary Material shows that $\Pi_{\varepsilon}\{a_{n,\varepsilon}(\theta-\theta_{\varepsilon})\in A\mid s_{{\rm obs}}\}$
and $\tPi_{\varepsilon}\{a_{n,\varepsilon}(\theta-\ttheta_{\varepsilon})\in A\mid s_{{\rm obs}}\}$
have the same limit, in distribution when $c_{\varepsilon}\in(0,\infty)$
and in total variation form when $c_{\varepsilon}=0$ or $\infty$.
Therefore it is sufficient to only consider the convergence of $\tPi_{\varepsilon}$
of the properly scaled and centered $\theta$. 

When $a_{n}\varepsilon_{n}\rightarrow c_{\varepsilon}<\infty$, according
to \eqref{eq:ABC_posterior_altform1}, $\tPi_{\varepsilon}\{a_{n}(\theta-\ttheta_{\varepsilon})\in A\mid s_{{\rm obs}}\}$ equals
\[
\int_{\mathbb{R}^{d}}\int_{t(B_{\delta})}\tPi\{a_{n}(\theta-\theta_{0})\in A+a_{n}(\ttheta_{\varepsilon}-\theta_{0})\mid s_{{\rm obs}}+\varepsilon_{n}v\}\tpi_{\varepsilon,tv}(t',v)^{({\rm norm})}\:dt'dv.
\]
By Lemma \ref{lem:misspecified_posterior_limit} and Lemma \ref{lem:ABC_probability_expand}(c),
we have
\begin{align*}
\sup_{A\in\mathscr{B}^{p}}\left|\tPi\{a_{n}(\theta-\theta_{0})\in A+a_{n}(\ttheta_{\varepsilon}-\theta_{0})\mid s_{{\rm obs}}+\varepsilon_{n}v\}-\int_{A}N\{t;\mu_{n}(v),I(\theta_{0})^{-1}\}\,dt\right|=o_{p}(1),
\end{align*}
where $\mu_{n}(v)=\beta_{0}\{c_{\varepsilon}v-a_{n}\varepsilon_{n}E_{G_{n}}(v)\}-a_{n}r_{n,1}$.
Then with Lemma \ref{lem:ABC_likelihood_expansion}, the leading
term of $\tPi_{\varepsilon}\{a_{n}(\theta-\ttheta_{\varepsilon}) \in A\mid s_{{\rm obs}}\}$
equals
\begin{align}
 & \sup_{A\in\mathscr{B}^{p}}\left|\tPi_{\varepsilon}\{a_{n}(\theta-\ttheta_{\varepsilon}) \in A \mid s_{{\rm obs}}\}-\int_{\mathbb{R}^{d}}\int_{t(B_{\delta})}\int_{A}N\{t;\mu_{n}(v),I(\theta_{0})^{-1}\}g_{n}(t',v)^{({\rm norm})}\,dtdt'dv\right|=o_{p}(1).\label{eq:ABC_posterior_leading}
\end{align}

The numerator of the leading term of \eqref{eq:ABC_posterior_leading} is in the form
\begin{align*}
\int_{\mathbb{R}^{d}}\int_{t(B_{\delta})}\int_{A}N\{t;c_{\varepsilon}\beta_{0}v+x_{3},I(\theta_{0})^{-1}\}N\{Ds(\theta_{0})t';x_{1}v+x_{2},A(\theta_{0})\}K(v)\,dtdt'dv,
\end{align*}
where $x_{1}\in\mathbb{R}$, $x_{2}\in\mathbb{R}^{d}$
and $x_{3}\in\mathbb{R}^{p}$. This is continuous by Lemma \ref{lem:continuity_of_integrals} in the Supplementary Material. Then
since $E_{G_{n}}(v)\rightarrow E_{G}(v)$ in distribution as $n\rightarrow\infty$ by Lemma
\ref{lem:ABC_likelihood_expansion}, we have
\begin{eqnarray*}
\lefteqn{\int_{\mathbb{R}^{d}}\int_{t(B_{\delta})}\int_{A}N\{t;\mu_{n}(v),I(\theta_{0})^{-1}\}g_{n}(t',v)\,dtdt'dv} \\
& & \rightarrow\int_{\mathbb{R}^{d}}\int_{\mathbb{R}^{p}}\int_{A}N[t;c_{\varepsilon}\beta_{0}\{v-E_{G}(v)\},I(\theta_{0})^{-1}]g(t',v)\,dtdt'dv,
\end{eqnarray*}
in distribution as $n\rightarrow\infty$. Putting the above results together, it
holds that 
\[
\tPi_{\varepsilon}\{a_{n}(\theta-\ttheta_{\varepsilon})\in A\mid s_{{\rm obs}}\}\rightarrow\int_{A}\int_{\mathbb{R}^{p}}N[t;c_{\varepsilon}\beta_{0}\{v-E_{G}(v)\},I(\theta_{0})^{-1}]G(v)^{({\rm norm})}\,dvdt,
\]
in distribution, and statement (ii) of the proposition holds. 

When $c_{\varepsilon}=0$, since $\mu_{n}(v)$ does not depend on
$v$, \eqref{eq:ABC_posterior_leading} becomes 
\[
\sup_{A\in\mathscr{B}^{p}}\left|\tPi_{\varepsilon}\{a_{n}(\theta-\ttheta_{\varepsilon})\in A\mid s_{{\rm obs}}\}-\int_{A}N\{t;-a_{n}\varepsilon_{n}\beta_{0}E_{G_{n}}(v)-a_{n}r_{n,1},I(\theta_{0})^{-1}\}\,dt\right|=o_{p}(1),
\]
and by the continuous mapping theorem \cite[]{van2000asymptotic}, 
\begin{align*}
 & \int_{\mathbb{R}^{p}}\left|N\{t;-a_{n}\varepsilon_{n}\beta_{0}E_{G_{n}}(v)-a_{n}r_{n,1},I(\theta_{0})^{-1}\}-N\{t;0,I(\theta_{0})^{-1}\}\right|\,dt=o_{p}(1).
\end{align*}
Therefore $\sup_{A\in\mathscr{B}^{p}}\left|\tPi_{\varepsilon}\{a_{n}(\theta-\ttheta_{\varepsilon})\in A\mid s_{{\rm obs}}\}-\int_{A}N\{t;0,I(\theta_{0})^{-1}\}\,dt\right|=o_{p}(1)$,
and statement (i) of the proposition holds. 

When $c_{\varepsilon}=\infty$, Lemma \ref{lem:misspecified_posterior_limit}
cannot be applied to the posterior distribution within \eqref{eq:ABC_posterior_altform1}
directly. With transformation $v''=a_{n}\varepsilon_{n}v$, $\tPi_{\varepsilon}\{\theta\in A\mid s_{{\rm obs}}\}$ equals, 
\[
\int_{\mathbb{R}^{d}}\int_{t(B_{\delta})}\tPi(\theta\in A\mid s_{{\rm obs}}+a_{n}^{-1}v'')\tpi_{\varepsilon,tv}(t',a_{n}^{-1}\varepsilon_{n}^{-1}v'')^{({\rm norm})}\,dt'dv'',
\]
which implies that $\tPi_{\varepsilon}\{\varepsilon_{n}^{-1}(\theta-\ttheta_{\varepsilon})\in A\mid s_{{\rm obs}}\}$ equals
\begin{align*}
& \int_{\mathbb{R}^{d}}\int_{t(B_{\delta})}\tPi\{a_{n}(\theta-\theta_{0})\in a_{n}\varepsilon_{n}A+a_{n}(\ttheta_{\varepsilon}-\theta_{0})\mid s_{{\rm obs}}+a_{n}^{-1}v''\}\tpi_{\varepsilon,tv}(t',a_{n}^{-1}\varepsilon_{n}^{-1}v'')^{({\rm norm})}\:dt'dv''.
\end{align*}
Then Lemma \ref{lem:misspecified_posterior_limit} can be applied.
Using Lemma \ref{lem:ABC_likelihood_expansion} and transforming $v''$
back to $v$ we have 
\begin{align}
 & \sup_{A\in\mathscr{B}^{p}}\left|\tPi_{\varepsilon}\{\varepsilon_{n}^{-1}(\theta-\ttheta_{\varepsilon})\in A\mid s_{{\rm obs}}\}\right.\nonumber \\
 & \left.-\int_{\mathbb{R}^{d}}\int_{t(B_{\delta})}\int_{A}(a_{n}\varepsilon_{n})^{p}N\{a_{n}\varepsilon_{n}t;\mu_{n}'(v),I(\theta_{0})^{-1}\}g_{n}(t',v)^{({\rm norm})}\:dtdt'dv\right|=o_{p}(1),\label{eq:ABC_posterior_leading2}
\end{align}
where $\mu_{n}'(v)=\beta_{0}\{A(\theta_{0})^{1/2}W_{{\rm obs}}+a_{n}\varepsilon_{n}v\}-a_{n}(\ttheta_{\varepsilon}-\theta_{0})$. 

Let $t''(t,t')=a_{n}\varepsilon_{n}(t-t')+a_{n}(\ttheta_{\varepsilon}-\theta_{0})$
and $t''(A,B_{\delta})$ be the set $\{t''(t,t'):t\in A,t'\in t(B_{\delta})\}$.
With transformations $v'=v'(v,t')$ and $t''=t''(t,t')$, since $\beta_{0}Ds(\theta_{0})=I_{p}$,
we have
\begin{align*}
 & \int_{\mathbb{R}^{d}}\int_{t(B_{\delta})}\int_{A}(a_{n}\varepsilon_{n})^{p}N\{a_{n}\varepsilon_{n}t;\mu_{n}'(v),I(\theta_{0})^{-1}\}g_{n}(t',v)\,dtdt'dv\\
 & =\int_{\mathbb{R}^{d}}\int_{t(B_{\delta})}\int_{A}(a_{n}\varepsilon_{n})^{p}N\{a_{n}\varepsilon_{n}(t-t');\beta_{0}v'-a_{n}(\ttheta_{\varepsilon}-\theta_{0}),I(\theta_{0})^{-1}\}g_{n}'(t',v')\,dtdt'dv'\\
 & =\int_{\mathbb{R}^{d}}\int_{t''(A,B_{\delta})}\int_{A}N\{t'';\beta_{0}v'-a_{n}(\ttheta_{\varepsilon}-\theta_{0}),I(\theta_{0})^{-1}\}g'_{n}\left[t-\frac{1}{a_{n}\varepsilon_{n}}\{t''-a_{n}(\ttheta_{\varepsilon}-\theta_{0})\},v'\right]\,dtdt''dv'.
\end{align*}
The idea now is that as $n\rightarrow \infty$, $a_n\varepsilon_n\rightarrow\infty$, so the $g'_n$ term in the integral will tend to $g'_n(t,v')$. Then by  integrating first with respect
to $t''$ and then with respect to $v$, we get the required result.

To make this argument rigorous, consider the following function, 
\begin{eqnarray*}
\lefteqn{\int_{\mathbb{R}^{d}}\int_{t''(A,B_{\delta})}\int_{A}N\{t'';\beta_{0}v',I(\theta_{0})^{-1}\}N\{v';0,A(\theta_{0})\}} \\
& & \times\left|K\{Ds(\theta_{0})t+x_{1}v'-x_{2}t''+x_{3}\}-K\{Ds(\theta_{0})t\}\right|\,dtdt''dv',
\end{eqnarray*}
where $x_{1}\in\mathbb{R}$, $x_{2}\in\mathbb{R}$
and $x_{3}\in\mathbb{R}^{d}$. This is continuous by Lemma \ref{lem:continuity_of_integrals} in the Supplementary Material, so by the continuous mapping theorem,
\begin{align*}
\sup_{A\in\mathscr{B}^{p}}\left|\int_{\mathbb{R}^{d}}\int_{t''(A,B_{\delta})}\int_{A}N\{t'';\beta_{0}v',I(\theta_{0})^{-1}\}g'_{n}\left(t-\frac{1}{a_{n}\varepsilon_{n}}t'',v'\right)\,dtdt''dv'-\int_{A}K\{Ds(\theta_{0})t\}\,dt\right| & =o_{p}(1).
\end{align*}
Then using Lemma \ref{lem:ABC_likelihood_expansion},
\[
\sup_{A\in\mathscr{B}^{p}}\left|\tPi_{\varepsilon}\{\varepsilon_{n}^{-1}(\theta-\ttheta_{\varepsilon})\in A\mid s_{{\rm obs}}\}-\int_{A}K\{Ds(\theta_{0})t\}\,
dt/\int_{\mathbb{R}^{p}}K\{Ds(\theta_{0})t\}\,dt\right|=o_{p}(1).
\]
Therefore statement (iii) of the proposition holds. 

\end{proof}

\subsection*{Proof of Result from Section \ref{adjABC_posterior}}

The intuition behind this result is that, as shown above, the joint posterior of $(\theta,s)$ under approximate Bayesian computation can be viewed as a marginal distribution for $s$ times a conditional for $\theta$
given $s$. The latter is just the true posterior for $\theta$ given $s$, and this posterior converges to a Gaussian limit that depends on $s$ only through its mean. Regression adjustment works because
it corrects for the dependence of this mean on $s$, so if we work with the regression adjusted parameter $\theta^*$ then the conditional distribution for $\theta^*$ given $s$ will tend to a Gaussian limit 
whose mean is the same for all $s$.

The following lemma gives an expansion of the optimal linear coefficient
matrix $\beta_{\varepsilon}$. The order of the remainder leads to
successful removal of the first order bias in the posterior distribution
of approximate Bayesian computation.

\begin{lemma} \label{lem:reg_coef}Assume Conditions \ref{cond:par_true}--\ref{cond:likelihood_moments}.
Then if $\varepsilon_{n}=o(a_{n}^{-3/5})$, $a_{n}\varepsilon_{n}(\beta_{\varepsilon}-\beta_{0})=o_{p}(1)$.
\end{lemma}

The following lemma, similar to Lemma \ref{lem:ABC_probability_expand},
says that the approximate Bayesian computation posterior distribution of $\theta^{*}$ is asymptotically
the same as $\tPi_{\varepsilon}$. Recall that $\theta_{\varepsilon}^{*}$
is the mean of $\pi_{\varepsilon}^{*}(\theta^{*}\mid s_{{\rm obs}})$.
Let $\tpi_{\varepsilon}^{*}(\theta^{*}\mid s_{{\rm obs}})=\int_{\mathbb{R}^{d}}\tpi_{\varepsilon}\{\theta^{*}+\beta_{\varepsilon}(s-s_{{\rm obs}}),s\mid s_{{\rm obs}}\}\,ds$
and $\ttheta_{\varepsilon}^{*}$ be the mean of $\tpi_{\varepsilon}^{*}(\theta^{*}\mid s_{{\rm obs}})$.
\begin{lemma} \label{lem:adjABC_probability_expand}Assume Conditions
\ref{cond:par_true}--\ref{cond:likelihood_moments}. If $\varepsilon_{n}=o(a_{n}^{-3/5})$,
then

(a) for any $\delta<\delta_{0}$, $\Pi_{\varepsilon}(\theta^{*}\in B_{\delta}^{c}\mid s_{{\rm obs}})$
and $\tPi_{\varepsilon}(\theta^{*}\in B_{\delta}^{c}\mid s_{{\rm obs}})$
are $o_{p}(1)$;

(b) there exists $\delta<\delta_0$ such that $\sup_{A\in\mathscr{B}^{p}}\left|\Pi_{\varepsilon}(\theta^{*}\in A\cap B_{\delta}\mid s_{{\rm obs}})-\tPi_{\varepsilon}(\theta^{*}\in A\cap B_{\delta}\mid s_{{\rm obs}})\right|=o_{p}(1)$;

(c) $a_{n}(\theta_{\varepsilon}^{*}-\ttheta_{\varepsilon}^{*})=o_{p}(1)$,
and $\ttheta_{\varepsilon}^{*}=\theta_{0}+a_{n}^{-1}\beta_{0}A(\theta_{0})^{1/2}W_{{\rm obs}}+\varepsilon_{n}(\beta_{0}-\beta_{\varepsilon})E_{G_{n}}(v)+r_{n,2}$
where the remainder $r_{n,2}=o_{p}(a_{n}^{-1})$. \end{lemma}

\begin{proof}[of Theorem \ref{thm:ABC_good_convergence}]Similar to the proof of Proposition \ref{thm:ABC_bad_convergence}, it is sufficient to only consider the convergence of $\tPi_{\varepsilon}$
of the properly scaled and centered $\theta^*$. Similar
to \eqref{eq:ABC_posterior_altform1},
\begin{align*}
\tPi_{\varepsilon}(\theta^{*}\in A\mid s_{{\rm obs}}) & =\begin{cases}
\int_{\mathbb{R}^{d}}\int_{t(B_{\delta})}\tPi(\theta\in A+\varepsilon_{n}\beta_{\varepsilon}v\mid s_{{\rm obs}}+\varepsilon_{n}v)\tpi_{\varepsilon,tv}(t',v)^{({\rm norm})}\,dt'dv, & \ c_{\varepsilon}<\infty,\\
\int_{\mathbb{R}^{d}}\int_{t(B_{\delta})}\tPi(\theta\in A+\varepsilon_{n}\beta_{\varepsilon}v\mid s_{{\rm obs}}+a_{n}^{-1}v)\tpi_{\varepsilon,tv}(t',a_{n}^{-1}\varepsilon_{n}^{-1}v)^{({\rm norm})}\,dt'dv, & \ c_{\varepsilon}=\infty.
\end{cases}
\end{align*}
Similar to \eqref{eq:ABC_posterior_leading}, by Lemma \ref{lem:misspecified_posterior_limit}
and Lemma \ref{lem:adjABC_probability_expand}(c), we have 
\begin{align}
 & \sup_{A\in\mathscr{B}^{p}}\left|\tPi_{\varepsilon}\{a_{n}(\theta^{*}-\ttheta_{\varepsilon}^{*})\in A\mid s_{{\rm obs}}\}-\int_{\mathbb{R}^{d}}\int_{t(B_{\delta})}\int_{A}N\{t;\mu_{n}^{*}(v),I(\theta_{0})^{-1}\}g_{n}(t',v)^{({\rm norm})}\,dtdt'dv\right|=o_{p}(1),\label{eq:adjABC_posterior_leading}
\end{align}
where $\mu_{n}^{*}(v)=\{a_{n}\varepsilon_{n}(\beta_{0}-\beta_{\varepsilon})+(c_{\varepsilon}-a_{n}\varepsilon_{n})\beta_{0}\mathbbm{1}_{c_{\varepsilon}<\infty}\}\{v-E_{G_{n}}(v)\}+a_{n}r_{n,2}$.
Since $a_{n}\varepsilon_{n}(\beta_{\varepsilon}-\beta_{0})=o_{p}(1)$,
by Lemma \ref{lem:continuity_of_integrals} in the Supplementary Material and the continuous mapping
theorem, 
\[
\int_{\mathbb{R}^{d}}\int_{t(B_{\delta})}\int_{\mathbb{R}^{p}}\left|N\{t;\mu_{n}^{*}(v),I(\theta_{0})^{-1}\}-N\{t;0,I(\theta_{0})^{-1}\}\right|g_{n}(t',v)\,dtdt'dv=o_{p}(1).
\]
Then we have
\begin{align*}
 & \sup_{A\in\mathscr{B}^{p}}\left|\tPi_{\varepsilon}\{a_{n}(\theta^{*}-\ttheta_{\varepsilon}^{*})\in A\mid s_{{\rm obs}}\}-\int_{A}N\{t;0,I(\theta_{0})^{-1}\}\,dt\right|=o_{p}(1),
\end{align*}
and the first convergence in the theorem holds.

The second convergence in the theorem holds by Lemma \ref{lem:adjABC_probability_expand}(c).

Since the only requirement for $\beta_{\varepsilon}$ in the above
is $a_{n}\varepsilon_{n}(\beta_{\varepsilon}^{'}-\beta_{\varepsilon})=o_{p}(1)$,
the above arguments will hold if $\beta_{\varepsilon}$ is replaced
by a $p\times d$ matrix $\hbeta_{\varepsilon}$ satisfying $a_{n}\varepsilon_{n}(\hbeta_{\varepsilon}-\beta_{\varepsilon})=o_{p}(1)$.
\end{proof}

\begin{proof}[of Proposition \ref{prop:ABC_good_convergence}]

Consider $\theta^{*}$ where $\beta_{\varepsilon}$ is replaced by
$\hbeta_{\varepsilon}$. Let $\eta_{n}=N^{1/2}a_{n}\varepsilon_{n}(\hbeta_{\varepsilon}-\beta_{\varepsilon})$.
Since $\hbeta_{\varepsilon}-\beta_{\varepsilon}=O_{p}\{(a_{n}\varepsilon_{n})^{-1}N^{-1/2}\}$
as $n\rightarrow\infty$, $\eta_{n}=O_{p}(1)$ as $n\rightarrow\infty$
and let its limit be $\eta$. In this case, if we replace $\mu_{n}^{*}(v)$
in \eqref{eq:adjABC_posterior_leading} with $\mu_{n}^{*}(v)+N^{-1/2}\eta_{n}\{v-E_{G_{n}}(v)\}$,
denoted by $\hmu_{n}^{*}(v)$, the equation still holds. Denote this
equation by $(7')$. Limits of the leading term in $(7')$ can be
obtained by arguments similar as those for \eqref{eq:ABC_posterior_leading}.

When $c_{\varepsilon}<\infty$, since for fixed $v$, $\hmu_{n}^{*}(v)$
converges to $N^{-1/2}\eta\{v-E_{G}(v)\}$ in distribution, by following
the same line we have 
\begin{align*}
\tPi_{\varepsilon}\{a_{n}(\theta^{*}-\ttheta_{\varepsilon}^{*})\in A\mid s_{{\rm obs}}\} & \rightarrow\int_{A}\int_{\mathbb{R}^{p}}N[t;N^{-1/2}\eta\{v-E_{G}(v)\},I(\theta_{0})^{-1}]G(v)^{({\rm norm})}\,dvdt,
\end{align*}
in distribution as $n\rightarrow\infty$.

When $c_{\varepsilon}=\infty$, by Lemma \ref{lem:ABC_likelihood_expansion}
we have $E_{G_{n}}(v)\rightarrow0$ in probability, and $\int_{\mathbb{R}^{d}}\int_{t(B_{\delta})}g_{n}(t',v)\,dt'dv\rightarrow\int_{\mathbb{R}^{p}}K\{Ds(\theta_{0})t\}\,dt$
in probability. Then with transformation $v'=v'(v,t')$, for fixed
$v$,
\begin{align*}
\hmu_{n}^{*}(v) & =\{a_{n}\varepsilon_{n}(\beta_{0}-\beta_{\varepsilon})+N^{-1/2}\eta_{n}\}\left\{ Ds(\theta_{0})t'+\frac{1}{a_{n}\varepsilon_{n}}v'-\frac{1}{a_{n}\varepsilon_{n}}A(\theta_{0})^{1/2}W_{{\rm obs}}-E_{G_{n}(v)}\right\} \\
 & \rightarrow N^{-1/2}\eta Ds(\theta_{0})t',
\end{align*}
in distribution as $n\rightarrow\infty$. Recall that $g_{n}(t',v)\,dv=g_{n}'(t',v')\,dv'$.
Then by Lemma \ref{lem:continuity_of_integrals} in the Supplementary Material and the continuous
mapping theorem,
\begin{align*}
 & \int_{\mathbb{R}^{d}}\int_{t(B_{\delta})}\int_{A}N\{t;\hmu_{n}^{*}(v),I(\theta_{0})^{-1}\}g_{n}'(t',v')\,dtdt'dv\\
\rightarrow & \int_{\mathbb{R}^{d}}\int_{t(B_{\delta})}\int_{A}N\{t;N^{-1/2}\eta Ds(\theta_{0})t',I(\theta_{0})^{-1}\}K\{Ds(\theta_{0})t'\}\,dtdt'dv,
\end{align*}
in distribution as $n\rightarrow\infty$. Therefore by $(7')$ and the above convergence results, 
\begin{align*}
\tPi_{\varepsilon}\{a_{n}(\theta^{*}-\ttheta_{\varepsilon}^{*})\in A\mid s_{{\rm obs}}\} & \rightarrow\int_{A}\int_{\mathbb{R}^{p}}N\{t;N^{-1/2}\eta Ds(\theta_{0})t',I(\theta_{0})^{-1}\}K\{Ds(\theta_{0})t'\}^{({\rm norm})}\,dt'dt,
\end{align*}
in distribution as $n\rightarrow\infty$.
\end{proof}

\section*{Supplementary Material}
\addcontentsline{toc}{section}{Supplementary Material}

\subsection*{Notations and Set-up}

First some limit notations and conventions are given. For
two sets $A$ and $B$, the sum of integrals $\int_{A}f(x)\,dx+\int_{B}f(x)\,dx$
is written as $(\int_{A}+\int_{B})f(x)\,dx$. For a constant $d\times p$
matrix $A$, let the minimum and maximum eigenvalues of $A^{T}A$
be $\lambda_{{\rm min}}^{2}(A)$ and $\lambda_{{\rm max}}^{2}(A)$
where $\lambda_{{\rm min}}(A)$ and $\lambda_{{\rm max}}(A)$ are
non-negative. Obviously, for any $p$-dimension vector $x$, $\lambda_{{\rm min}}(A)\|x\|\leq\|Ax\|\leq\lambda_{{\rm max}}(A)\|x\|$.
For two matrices $A$ and $B$, we say $A$ is bounded by $B$ or
$A\leq B$ if $\lambda_{{\rm max}}(A)\leq\lambda_{{\rm min}}(B)$.
For a set of matrices $\{A_{i}:i\in I\}$ for some index set $I$,
we say it is bounded if $\lambda_{{\rm max}}(A_{i})$ are uniformly
bounded in $i$. Denote the identity matrix with dimension $d$ by
$I_{d}$. Notations from the main text will also be used.

The following basic asymptotic results \cite[]{serfling2009approximation}
will be used throughout. \begin{lemma} \label{lem:basic_asymp} (i) For a series of random
variables $Z_{n}$, if $Z_{n}\rightarrow Z$ in distribution as $n\rightarrow\infty$, $Z_{n}=O_{p}(1)$.
(ii) (Continuous mapping) For a series of continuous function $g_{n}(x)$, if $g_{n}(x)=O(1)$
almost everywhere, then $g_{n}(Z_{n})=O_{p}(1)$, and this also holds if
$O(1)$ and $O_{p}(1)$ are replaced by $\Theta(1)$ and $\Theta_{p}(1)$.
\end{lemma}

Some notations regarding the posterior distribution of approximate Bayesian computation are given. For $A\subset\mathbbm{R}^{p}$
and a scalar function $h(\theta,s)$, let 
\[\pi_{A}(h)=\int_{A}\int_{\mathbb{R}^{d}}h(\theta,s)\pi(\theta)f_{n}(s\mid\theta)K\{\varepsilon_{n}^{-1}(s-s_{{\rm obs}})\}\varepsilon_{n}^{-d}\,dsd\theta,\]
and 
\[\tpi_{A}(h)=\int_{A}\int_{\mathbb{R}^{d}}h(\theta,s)\pi_{\delta}(\theta)\ftil_{n}(s\mid\theta)K\{\varepsilon_{n}^{-1}(s-s_{{\rm obs}})\}\varepsilon_{n}^{-d}\,dsd\theta.\]
Then $\Pi_{\varepsilon}(\theta\in A\mid s_{{\rm obs}})=\pi_{A}(1)/\pi_{\mathcal{P}}(1)$
and its normal counterpart $\tPi_{\varepsilon}(\theta\in A\mid s_{{\rm obs}})=\tpi_{A}(1)/\tpi_{\mathcal{P}}(1)$.

The following results from \citet{Li/Fearnhead:2018} will be used throughout. 

\begin{lemma} \label{lem:LF_results} Assume Conditions \ref{cond:par_true}\textendash \ref{cond:sum_approx}.
Then as $n\rightarrow\infty$,

(i) if Condition \ref{cond:sum_approx_tail} also holds then, for any $\delta<\delta_{0}$, $\pi_{B_{\delta}^{c}}(1)$ and $\tpi_{B_{\delta}^{c}}(1)$
are $o_{p}(1)$, and $O_{p}(e^{-a_{n,\varepsilon}^{\alpha_{\delta}}c_{\delta}})$ for
some positive constants $c_{\delta}$ and $\alpha_{\delta}$ depending
on $\delta$;

(ii) $\pi_{B_{\delta}}(1)=\tpi_{B_{\delta}}(1)\{1+O_{p}(\alpha_{n}^{-1})\}$
and $\mbox{{\rm sup}}_{A\subset B_{\delta}}\left|\pi_{A}(1)-\tpi_{A}(1)\right|/\tpi_{B_{\delta}}(1)=O_{p}(\alpha_{n}^{-1})$;

(iii) if $\varepsilon_{n}=o(a_{n}^{-1/2})$, $\tpi_{B_{\delta}}(1)$
and $\pi_{B_{\delta}}(1)$ are $\Theta_{p}(a_{n,\varepsilon}^{d-p})$, and thus $\tpi_{\mathcal{P}}(1)$ and $\pi_{\mathcal{P}}(1)$ are
$\Theta_{p}(a_{n,\varepsilon}^{d-p})$;

(iv) if $\varepsilon_{n}=o(a_{n}^{-1/2})$ and Condition \ref{cond:sum_approx_tail} holds, $\theta_{\varepsilon}=\ttheta_{\varepsilon}+o_{p}(a_{n,\varepsilon}^{-1})$.
If $\varepsilon_{n}=o(a_{n}^{-3/5})$, $\theta_{\varepsilon}=\ttheta_{\varepsilon}+o_{p}(a_{n}^{-1})$.

\end{lemma}

\begin{proof} (i) is from \citet[Lemma 3]{Li/Fearnhead:2018} and a
trivial modification of its proof when Condition \ref{cond:sum_approx_tail} does no hold; (ii) is from \citet[equation 13 of supplements]{Li/Fearnhead:2018};
(iii) is from \citet[Lemma 5 and equation 13 of supplements]{Li/Fearnhead:2018}; and  
(iv) is from \citet[Lemma 3 and Lemma 6]{Li/Fearnhead:2018}. \end{proof} \medskip{}

\subsection*{Proof for Results in Section \ref{ABC_posterior}}

\begin{proof}[of Lemma \ref{lem:misspecified_posterior_limit}]

For any fixed $v\in\mathbb{R}^{d}$, recall that $\tPi(\theta\in A\mid s_{{\rm obs}}+\varepsilon_{n}v)$
is the posterior distribution given $s_{{\rm obs}}+\varepsilon_{n}v$
with prior $\pi_{\delta}(\theta)$ and the misspecified model $\widetilde{f}_{n}(\cdot\mid\theta)$.
By \cite{kleijn2012bernstein}, if there exist $\Delta_{n,\theta_{0}}$
and $V_{\theta_{0}}$ such that,
\begin{description}
\item [{\rm (KV1)}] for any compact set $K\subset t(B_{\delta})$, 
\begin{align*}
\sup_{t\in K}\left|\log\frac{\widetilde{f}_{n}(s_{{\rm obs}}+\varepsilon_{n}v\mid\theta_{0}+a_{n}^{-1}t)}{\widetilde{f}_{n}(s_{{\rm obs}}+\varepsilon_{n}v\mid\theta_{0})}-t^{T}V_{\theta_{0}}\Delta_{n,\theta_{0}}+\frac{1}{2}t^{T}V_{\theta_{0}}t\right| & \rightarrow0,
\end{align*}
 in probability as $n\rightarrow\infty$, and
\item [{\rm (KV2)}] $E\{\tPi(a_{n}\|\theta-\theta_{0}\|>M_{n}\mid s_{{\rm obs}}+\varepsilon_{n}v)\}\rightarrow0$  as $n\rightarrow\infty$ for any sequence of constants $M_{n}\rightarrow\infty$,
\end{description}
then
\[
\sup_{A\in\mathscr{B}^{p}}\left|\tPi\{a_{n}(\theta-\theta_{0})\in A\mid s_{{\rm obs}}+\varepsilon_{n}v\}-\int_{A}N(t;\Delta_{n,\theta_{0}},V_{\theta_{0}}^{-1})\,dt\right|\rightarrow0,
\]
in probability as $n\rightarrow\infty$. 

For (KV1), by the definition of $\widetilde{f}_{n}(s\mid\theta)$,
\[
\log\frac{\widetilde{f}_{n}(s_{{\rm obs}}+\varepsilon_{n}v\mid\theta_{0}+a_{n}^{-1}t)}{\widetilde{f}_{n}(s_{{\rm obs}}+\varepsilon_{n}v\mid\theta_{0})}=\log\frac{N\{s_{{\rm obs}}+\varepsilon_{n}v;s(\theta_{0}+a_{n}^{-1}t),a_{n}^{-2}A(\theta_{0}+a_{n}^{-1}t)\}}{N\{s_{{\rm obs}}+\varepsilon_{n}v;s(\theta_{0}),a_{n}^{-2}A(\theta_{0})\}}.
\]
As $x^{T}Ax-y^{T}By=x^{T}(A-B)x+(x-y)^{T}B(x+y)$, for
vectors $x$ and $y$ and matrices $A$ and $B$, by applying a Taylor expansion
on $s(\theta_{0}+xt)$ and $A(\theta_{0}+xt)$ around $x=0$, the
right hand side of above equation equals
\begin{align*}
 & \{Ds(\theta_{0}+e_{n}^{(1)}t)t\}^{T}A(\theta_{0})^{-1}\zeta_{n}(v,t)-\frac{a_{n}^{-1}}{2}\zeta_{n}(v,t)^{T}\left\{ \sum_{i=1}^{p}D_{\theta_{i}}A^{-1}(\theta_{0}+e_{n}^{(2)}t)t_{i}\right\} \zeta_{n}(v,t)\\
 & +\frac{a_{n}^{-1}}{2}\left\{D\log\left|A(\theta_{0}+e_{n}^{(3)}t)\right|\right\}^{T}t,
\end{align*}
where $\zeta_{n}(v,t)=A(\theta_{0})^{1/2}W_{{\rm obs}}+a_{n}\varepsilon_{n}v-\frac{1}{2}Ds(\theta_{0}+e_{n}^{(1)}t)t$
and for $j=1,2,3$, $e_{n}^{(j)}$ is a function of $t$ satisfying
$|e_{n}^{(j)}|\le a_{n}^{-1}$ which is from the remainder of the Taylor
expansions. Since $Ds(\theta)$, $DA^{-1}(\theta)$
and $D\log\left|A(\theta)\right|$ are bounded in $B_{\delta}$ when $\delta$ is small enough, 
\[
\sup_{t\in K}\left|\log\frac{\widetilde{f}_{n}(s_{{\rm obs}}+\varepsilon_{n}v\mid\theta_{0}+a_{n}^{-1}t)}{\widetilde{f}_{n}(s_{{\rm obs}}+\varepsilon_{n}v\mid\theta_{0})}-t^{T}I(\theta_{0})\beta_{0}\{A(\theta_{0})^{1/2}W_{{\rm obs}}+c_{\varepsilon}v\}+\frac{1}{2}t^{T}I(\theta_{0})t\right|\rightarrow0,
\]
 in probability as $n\rightarrow\infty$, for any compact set $K$. Therefore (KV1) holds with
$\Delta_{n,\theta_{0}}=\beta_{0}\{A(\theta_{0})^{1/2}W_{{\rm obs}}+c_{\varepsilon}v\}$
and $V_{\theta_{0}}=I(\theta_{0})$. 

For (KV2), let $r_{n}(s\mid\theta_{0})=\alpha_{n}\{f_{n}(s\mid\theta_{0})-\ftil_{n}(s\mid\theta_{0})\}$.
Since $r_{n}(s\mid\theta_{0})$ is bounded by a function integrable in
$\mathbb{R}^{d}$ by Condition \ref{cond:sum_approx},
\begin{align*}
 & E\{\tPi(a_{n}\|\theta-\theta_{0}\|>M_{n}\mid s_{{\rm obs}}+\varepsilon_{n}v)\}-\int_{\mathbb{R}^{d}}\tPi(a_{n}\|\theta-\theta_{0}\|>M_{n}\mid s+\varepsilon_{n}v)\ftil_{n}(s\mid \theta_{0})\,ds\\
\le & \alpha_{n}^{-1}\int_{\mathbb{R}^{d}}|r_{n}(s\mid \theta_{0})|\,ds=o(1).
\end{align*}
Then it is sufficient for the expectation under $\ftil_{n}(s\mid \theta_{0})$
to be $o(1)$. For any constant $M>0$, with the transformation $\bar{v}=a_{n}\{s-s(\theta_{0})\}$,
\begin{align*}
 & \int_{\mathbb{R}^{d}}\tPi(a_{n}\|\theta-\theta_{0}\|>M_{n}\mid s+\varepsilon_{n}v)\ftil_{n}(s\mid \theta_{0})\,ds\\
\leq & \int_{\|\bar{v}\|\le M}\frac{\int_{\|t\|>M_{n}}\tpi(t,\bar{v}\mid v)\,dt}{\int_{t(B_{\delta})}\tpi(t,\bar{v}\mid v)\,dt}N\{\bar{v};0,A(\theta_{0})\}\,d\bar{v}+\int_{\|\bar{v}\|>M}N\{\bar{v};0,A(\theta_{0})\}\,d\bar{v},
\end{align*}
where $\tpi(t,\bar{v}\mid v)=\pi_{\delta}(\theta_{0}+a_{n}^{-1}t)\ftil_{n}\{s(\theta_{0})+a_{n}^{-1}\bar{v}+\varepsilon_{n}v\mid \theta_{0}+a_{n}^{-1}t\}$.
For the first term in the above upper bound, it is bounded by a series which does not depend
on $M$ and is $o(1)$ as $M_{n}\rightarrow\infty$, as shown below. Obviously $\int_{t(B_{\delta})}\tpi(t,\bar{v}\mid v)\,dt$
can be lower bounded for some constant $m_{\delta}>0$. Choose $\delta$
small enough such that $Ds(\theta)$ and $A(\theta)^{1/2}$ are bounded
for $\theta\in B_{\delta}$. Let $\lambda_{{\rm min}}$ and $\lambda_{{\rm max}}$
be their common bounds. When $\|\bar{v}\|<M$ and $M_{n}$ is large
enough,
\begin{align}
\{t:\|t\|>M_{n}\} & \subset\left\{ t:\frac{\sup_{\theta\in B_{\delta}}\|Ds(\theta)t\|}{2}\geq\|a_{n}\varepsilon_{n}v+\bar{v}\|\right\} .\label{eq:set_t}
\end{align}
Then since for any $\bar{v}$ satisfying $\|\bar{v}\|<M$, by a Taylor
expansion,
\begin{align*}
\ftil_{n}\{s(\theta_{0})+a_{n}^{-1}\bar{v}+\varepsilon_{n}v\mid \theta_{0}+a_{n}^{-1}t\} & =a_{n}^{d}N\{Ds(\theta_{0}+e_{n}^{(1)}t)t;\bar{v}+a_{n}\varepsilon_{n}v,A(\theta_{0}+a_{n}^{-1}t)\},
\end{align*}
$\tpi(t,\bar{v}\mid v)\leq cN(\lambda_{{\rm max}}^{-1}\lambda_{{\rm min}}\|t\|/2;0,1)$, where
$c$ is some positive constant, for $t$ in the right hand side of \eqref{eq:set_t}. Then 
\[
\int_{\|\bar{v}\|\le M}\frac{\int_{\|t\|>M_{n}}\tpi(t,\bar{v}\mid v)\,dt}{\int_{t(B_{\delta})}\tpi(t,\bar{v}\mid v)\,dt}N\{\bar{v};0,A(\theta_{0})\}\,d\bar{v}\leq m_{\delta}^{-1}c\int_{\|t\|>M_{n}}N(\lambda_{{\rm max}}^{-1}\lambda_{{\rm min}}\|t\|/2;0,1)\,dt,
\]
the right hand side of which is $o(1)$ when $M_{n}\rightarrow\infty$.
Meanwhile by letting $M\rightarrow\infty$, it can be seen that the
expectation under $\ftil_{n}(s\mid \theta_{0})$ is $o(1)$. Therefore
(KV2) holds and the lemma holds.
\end{proof}

The following lemma is used for equations $\int_{\mathbb{R}^{p}}g_{n}(t,v)\,dt=|A(\theta_{0})|^{-1/2}G_{n}(v)$ and $\int_{\mathbb{R}^{p}}g(t,v)\,dt=|A(\theta_{0})|^{-1/2}G(v)$.
\begin{lemma}\label{lem:norm_dens_expand}

For a rank-$p$ $d\times p$ matrix $A$, a rank-$d$ $d\times d$
matrix $B$ and a $d$-dimension vector $c$,
\begin{align}
N(At;Bv+c,I_{d}) & =N\left\{t;(A^{T}A)^{-1}A^{T}(c+Bv),(A^{T}A)^{-1}\right\}g(v;A,B,c),\label{eq:normal_decomp}
\end{align}
where $P=A^{T}A$, and 
\begin{align*}
 & g(v;A,B,c)=\frac{1}{(2\pi)^{(d-p)/2}}\exp\left\{-\frac{1}{2}(c+Bv)^{T}(I-A(A^{T}A)^{-1}A^{T})(c+Bv)\right\}.
\end{align*}

\end{lemma}

\begin{proof}
This can be verified easily by matrix algebra.
\end{proof}

The following lemma regarding the continuity of a certain form of
integral will be helpful when applying the continuous mapping theorem.

\begin{lemma}\label{lem:continuity_of_integrals} Let $l_{1}$, $l_{1}'$,
$l_{2}$, $l_{2}'$ and $l_{3}$ be positive integers satisfying $l_{1}'\leq l_{1}$
and $l_{2}'\leq l_{2}$. Let $A$ and $B$ be $l_{1}\times l_{1}'$
and $l_{2}\times l_{2}'$ matrices, respectively, satisfying that
$A^{T}A$ and $B^{T}B$ are positive definite. Let $g_{1}(\cdot)$,
$g_{2}(\cdot)$ and $g_{3}(\cdot)$ be functions in $\mathbb{R}^{l_{1}}$,
$\mathbb{R}^{l_{2}}$ and $\mathbb{R}^{l_{3}}$, respectively, that
are integrable and continuous almost everywhere. Assume:

(i) $g_{j}(\cdot)$ is bounded in $\mathbb{R}^{l_{j}}$ for $j=1,2$;

(ii) $g_{j}(w)$ depends on $w$ only through $\|w\|$ and is a decreasing
function of $\|w\|$, for $j=1,2$; and

(iii) there exists a non-negative integer $l$ such that $\int_{\mathbb{R}^{l_{3}}}\prod_{k=1}^{l_{1}'+l_{2}'+l}w_{i_{k}}g_{3}(w)\,dw<\infty$
for any coordinates $(w_{i_{1}},\ldots,w_{i_{l_{1}'+l_{2}'+l}})$
of $w$. 

\noindent
Then the function,
\begin{eqnarray*}
\int\int\int P_{l}(w_{1},w_{2},w_{3})\left|g_{1}(Aw_{1}+x_{1}w_{2}+x_{2}w_{3}+x_{3})-g_{1}(Aw_{1})\right| g_{2}(Bw_{2}+x_{4}w_{3}+x_{5})g_{3}(w_{3})\,dw_{3}dw_{2}dw_{1},
\end{eqnarray*}
where $x_{1}\in\mathbb{R}^{l_{1}\times l_{2}'}$, $x_{2}\in\mathbb{R}^{l_{1}\times l_{3}}$, $x_{4}\in\mathbb{R}^{l_{2}\times l_{3}}$, $x_{3}\in\mathbb{R}^{l_{1}}$ and $x_{5}\in\mathbb{R}^{l_{2}}$, is continuous almost everywhere. 

\end{lemma}

\begin{proof}
Let $m_{A}$ and $m_{B}$ be the lower bound of $A$ and $B$ respectively.
For any $(x_{01,}\ldots,x_{05})\in\mathbb{R}^{l_{1}\times l_{2}'}\times\mathbb{R}^{l_{1}\times l_{3}}\times\mathbb{R}^{l_{2}\times l_{3}}\times\mathbb{R}^{l_{1}}\times\mathbb{R}^{l_{2}}$
such that the integrand in the target integral is continuous, consider
any sequence $(x_{n1},\ldots,x_{n5})$ converging to $(x_{01,}\ldots,x_{05})$.
It is sufficient to show the convergence of the target function at
$(x_{n1},\ldots,x_{n5})$. Let $V_{A}=\{w_{1}:\|Aw_{1}\|/2\geq\sup_{(x_{n1},x_{n2},x_{n3})}\|x_{n1}w_{2}+x_{n2}w_{3}+x_{n3}\|\}$,
$V_{B}=\{w_{2}:\|Bw_{2}\|/2\geq\sup_{(x_{n4},x_{n5})}\|x_{n4}w_{3}+x_{n5}\|\}$,
$U_{A}=\{w_{1}:\|w_{1}\|\leq4m_{A}^{-1}(\|x_{01}w_{2}\|+\|x_{02}w_{3}\|+\|x_{03}\|)\}$
and $U_{B}=\{w_{2}:\|w_{2}\|\leq4m_{B}^{-1}(\|x_{04}w_{3}\|+\|x_{05}\|)\}$.
We have $V_{A}^{c}\subset U_{A}$ and $V_{B}^{c}\subset U_{B}$. Then
according to the following upper bounds and condition (iii), 

\begin{align*}
 & \left|g_{1}(Aw_{1}+x_{n1}w_{2}+x_{n2}w_{3}+x_{n3})-g_{1}(Aw_{1})\right|\leq g_{1}(Aw_{1}+x_{n1}w_{2}+x_{n2}w_{3}+x_{n3})+g_{1}(Aw_{1}),\\
 & g_{1}(Aw_{1}+x_{n1}w_{2}+x_{n2}w_{3}+x_{n3})\leq\bar{g}_{1}(m_{A}\|w_{1}\|/2)\mathbbm{1}_{\{w_{1}\in V_{A}\}}+\sup_{w\in\mathbb{R}^{l_{1}}}g_{1}(w)\mathbbm{1}_{\{w_{1}\in U_{A}\}},\\
 & g_{2}(Bw_{2}+x_{4}w_{3}+x_{5})\leq\bar{g}_{2}(m_{B}\|w_{2}\|/2)\mathbbm{1}_{\{w_{2}\in V_{B}\}}+\sup_{w\in\mathbb{R}^{l_{2}}}g_{2}(w)\mathbbm{1}_{\{w_{2}\in U_{B}\}},
\end{align*}
where $g_{1}(w)=\bar{g}_{1}(\|w\|)$ and $g_{2}(w)=\bar{g}_{2}(\|w\|)$,
by applying the dominated convergence theorem, the target function
at $(x_{n1,}\ldots,x_{n5})$ converges to its value at $(x_{01,}\ldots,x_{05})$.
\end{proof}

\begin{proof}[of Lemma \ref{lem:ABC_likelihood_expansion}] The first
part holds according to Lemma 5 of \cite{Li/Fearnhead:2018}.
For the second part, when $c_{\varepsilon}=\infty$, by the transformation $v'=v'(v,t)$,
\begin{align*}
\int_{\mathbb{R}^{d}}\int_{t(B_{\delta})}P_{l}(v)g_{n}(t,v)\,dtdv & =\int_{\mathbb{R}^{d}}\int_{t(B_{\delta})}P_{l}\left\{Ds(\theta_{0})t+\frac{1}{a_{n}\varepsilon_{n}}v'-\frac{1}{a_{n}\varepsilon_{n}}A(\theta_{0})^{1/2}W_{{\rm obs}}\right\}g_{n}'(t,v')\,dtdv'.
\end{align*}
By applying Lemma \ref{lem:continuity_of_integrals} and the
continuous mapping theorem in Lemma \ref{lem:basic_asymp} to the right hand
side of the above when $c_{\varepsilon}=\infty$, and to $\int_{\mathbb{R}^{d}}\int_{t(B_{\delta})}P_{l}(v)g_{n}(t,v)\,dtdv$
when $c_{\varepsilon}<\infty$, and using $\int_{\mathbb{R}^{p}}g(t,v)\,dt=|A(\theta_{0})|^{-1/2}G(v)$,
the lemma holds. 
\end{proof}

\begin{proof}[of Lemma \ref{lem:ABC_probability_expand}] (a), (b)
and the first part of (c) hold immediately by Lemma \ref{lem:LF_results}.
The second part of (c) is stated in the proof of Theorem 1 of  \cite{Li/Fearnhead:2018}.
\end{proof}

\begin{lemma}\label{lem:normal_leading_for_scaled_theta} Assume
conditions \ref{cond:par_true}--\ref{cond:sum_approx_tail}. 

(i) If $c_{\varepsilon}\in(0,\infty)$ then $\Pi_{\varepsilon}\{a_{n}(\theta-\theta_{\varepsilon})\in A\mid s_{{\rm obs}}\}$
and $\tPi_{\varepsilon}\{a_{n}(\theta-\ttheta_{\varepsilon})\in A\mid s_{{\rm obs}}\}$
have the same limit in distribution.

(ii) If $c_{\varepsilon}=0$ or $c_{\varepsilon}=0\infty$ then
\[\mbox{\rm sup}_{A\in\mathscr{B}^{p}}\left|\Pi_{\varepsilon}\{a_{n,\varepsilon}(\theta-\theta_{\varepsilon})\in A\mid s_{{\rm obs}}\}-\tPi_{\varepsilon}\{a_{n,\varepsilon}(\theta-\ttheta_{\varepsilon})\in A\mid s_{{\rm obs}}\}\right|=o_{p}(1).\]

(iii)  If Condition \ref{cond:likelihood_moments} holds then 
\[\mbox{{\rm sup}}_{A\in\mathscr{B}^{p}}\left|\Pi_{\varepsilon}\{a_{n}(\theta^{*}-\theta_{\varepsilon}^{*})\in A\mid s_{{\rm obs}}\}
-\tPi_{\varepsilon}\{a_{n}(\theta^{*}-\ttheta_{\varepsilon}^{*})\in A\mid s_{{\rm obs}}\}\right|=o_{p}(1).\]

\end{lemma}

\begin{proof}

Let $\lambda_{n}=a_{n,\varepsilon}(\theta_{\varepsilon}-\ttheta_{\varepsilon})$,
and by Lemma \ref{lem:ABC_probability_expand}(c), $\lambda_{n}=o_{p}(1)$.
When $c_{\varepsilon}\in(0,\infty)$, for any $A\in\mathscr{B}^{p}$,
decompose $\Pi_{\varepsilon}\{a_{n}(\theta-\theta_{\varepsilon})\in A\mid s_{{\rm obs}}\}$
into the following three terms, 
\begin{align*}
 & \left[\Pi_{\varepsilon}\{a_{n}(\theta-\theta_{\varepsilon})\in A\mid s_{{\rm obs}}\}-\tPi_{\varepsilon}\{a_{n}(\theta-\theta_{\varepsilon})\in A\mid s_{{\rm obs}}\}\right]\\
+ & \left[\tPi_{\varepsilon}\{a_{n}(\theta-\ttheta_{\varepsilon})\in A+\lambda_{n}\mid s_{{\rm obs}}\}-\tPi_{\varepsilon}\{a_{n}(\theta-\ttheta_{\varepsilon})\in A\mid s_{{\rm obs}}\}\right]\\
+ & \tPi_{\varepsilon}\{a_{n}(\theta-\ttheta_{\varepsilon})\in A\mid s_{{\rm obs}}\}.
\end{align*}
For (i) to hold, it is sufficient that the first two terms in the
above are $o_{p}(1)$. The first term is $o_{p}(1)$ by Lemma \ref{lem:ABC_probability_expand}.
For the second term to be $o_{p}(1)$, given the leading term of $\tPi_{\varepsilon}\{a_{n}(\theta-\ttheta_{\varepsilon})\in A\mid s_{{\rm obs}}\}$
stated in the proof of Proposition \ref{thm:ABC_bad_convergence}
in the main text, it is sufficient that 
\begin{align*}
 & \sup_{v\in\mathbb{R}^{d}}\left|\left(\int_{A+\lambda_{n}}-\int_{A}\right)N\{t;\mu_{n}(v),I(\theta_{0})^{-1}\}\,dt\right|=o_{p}(1).
\end{align*}
This holds by noting that the left hand side of the above is bounded
by $(\int_{A+\lambda_{n}}-\int_{A})c\,dt$ for some constant $c$
and this upper bound is $o_{p}(1)$ since $\lambda_{n}=o_{p}(1)$.
Therefore (i) holds.

When $c_{\varepsilon}=0$ or $\infty$, $\mbox{{\rm sup}}_{A\in\mathscr{B}^{p}}\left|\Pi_{\varepsilon}\{a_{n,\varepsilon}(\theta-\theta_{\varepsilon})\in A\mid s_{{\rm obs}}\}-\tPi_{\varepsilon}\{a_{n,\varepsilon}(\theta-\ttheta_{\varepsilon})\in A\mid s_{{\rm obs}}\}\right|$
is bounded by
\begin{align}
 & \mbox{{\rm sup}}_{A\in\mathscr{B}^{p}}\left|\Pi_{\varepsilon}\{a_{n,\varepsilon}(\theta-\theta_{\varepsilon})\in A\mid s_{{\rm obs}}\}-\tPi_{\varepsilon}\{a_{n,\varepsilon}(\theta-\theta_{\varepsilon})\in A\mid s_{{\rm obs}}\}\right|\nonumber \\
+ & \mbox{{\rm sup}}_{A\in\mathscr{B}^{p}}\left|\tPi_{\varepsilon}\{a_{n,\varepsilon}(\theta-\ttheta_{\varepsilon})\in A+\lambda_{n}\mid s_{{\rm obs}}\}-\int_{A+\lambda_{n}}\psi(t)\,dt\right|\nonumber \\
+ & \mbox{{\rm sup}}_{A\in\mathscr{B}^{p}}\left|\tPi_{\varepsilon}\{a_{n,\varepsilon}(\theta-\ttheta_{\varepsilon})\in A\mid s_{{\rm obs}}\}-\int_{A}\psi(t)\,dt\right|\nonumber \\
+ & \mbox{{\rm sup}}_{A\in\mathscr{B}^{p}}\left|\int_{A+\lambda_{n}}\psi(t)\,dt-\int_{A}\psi(t)\,dt\right|.\label{eq:decomposed_upper_bound}
\end{align}
With similar arguments as before, the first three terms are $o_{p}(1)$.
For the fourth term, by transforming $t$ to $t+\lambda_{n}$, it
is upper bounded by $\int_{\mathbb{R}^{p}}|\psi(t-\lambda_{n})-\psi(t)|\,dt$
which is $o_{p}(1)$ by the continuous mapping theorem. Therefore
(ii) holds.

For (iii), the left hand side of the equation has the decomposed upper
bound similar to \eqref{eq:decomposed_upper_bound}, with $\theta$,
$\theta_{\varepsilon}$, $\ttheta_{\varepsilon}$ and $\psi(t)$ replaced
by $\theta^{*}$, $\theta_{\varepsilon}^{*}$, $\ttheta_{\varepsilon}^{*}$
and $N\{t;0,I(\theta_{0})^{-1}\}$. Then by Lemma \ref{lem:adjABC_probability_expand},
using the leading term of $\Pi_{\varepsilon}\{a_{n}(\theta^{*}-\theta_{\varepsilon}^{*})\in A\mid s_{{\rm obs}}\}$
stated in the proof of Theorem \ref{thm:ABC_good_convergence}, and
similar arguments to those used for the fourth term of \eqref{eq:decomposed_upper_bound},
it can be seen that this upper bound is $o_{p}(1)$. Therefore (iii)
holds.
\end{proof}

\subsection*{Proof for Results in Section \ref{adjABC_posterior}}

To prove Lemmas \ref{lem:reg_coef} and \ref{lem:adjABC_probability_expand},
some notation regarding the regression adjusted approximate Bayesian computation posterior, similar
to those defined previously, are needed. Consider transformations
$t=t(\theta)$ and $v=v(s)$. For $A\subset\mathbbm{R}^{p}$ and the
scalar function $h(t,v)$ in $\mathbb{R}^{p}\times\mathbb{R}^{d}$,
let $\tpi_{A,tv}(h)=\int_{t(A)}\int_{\mathbb{R}^{d}}h(t,v)\tpi_{\varepsilon,tv}(t,v)\,dvdt$.

\begin{proof}[of Lemma \ref{lem:reg_coef}]

Since $\beta_{\varepsilon}=\mbox{cov}_{\varepsilon}(\theta,s)\mbox{var}_{\varepsilon}(s)^{-1}$,
to evaluate the covariance matrices, we need to evaluate $\pi_{\mathbb{R}^{p}}\{(\theta-\theta_{0})^{k_{1}}(s-s_{{\rm obs}})^{k_{2}}\}/\pi_{\mathbb{R}^{p}}(1)$
for $(k_{1},k_{2})=(0,0),$ $(1,0),$ $(1,1),$ $(0,1)$ and $(0,2)$.

First of all, we show that $\pi_{B_{\delta}^{c}}\{(\theta-\theta_{0})^{k_{1}}(s-s_{{\rm obs}})^{k_{2}}\}$
is ignorable for any $\delta<\delta_{0}$ by showing that it is $O_{p}(e^{-a_{n,\varepsilon}^{\alpha_{\delta}}c_{\delta}})$
for some positive constants $c_{\delta}$ and $\alpha_{\delta}$.
By dividing $\mathbb{R}^{d}$ into $\{v:\|\varepsilon_{n}v\|\leq\delta'/3\}$
and its complement,
\begin{align}
 & \sup_{\theta\in B_{\delta}^{c}}\int_{\mathbb{R}^{d}}(s-s_{{\rm obs}})^{k_{2}}f_{n}(s\mid\theta)K\left(\frac{s-s_{{\rm obs}}}{\varepsilon_{n}}\right)\varepsilon_{n}^{-d}\,ds\nonumber \\
\leq & \sup_{\theta\in B_{\delta}^{c}}\left\{ \sup_{\|s-s_{{\rm obs}}\|\leq\delta'/3}f_{n}(s\mid\theta)\int_{\mathbb{R}^{d}}(s-s_{{\rm obs}})^{k_{2}}K\left(\frac{s-s_{{\rm obs}}}{\varepsilon_{n}}\right)\varepsilon_{n}^{-d}\,ds\right\} \nonumber \\
 & +\overline{K}\{\lambda_{{\rm min}}(\Lambda)\varepsilon_{n}^{-1}\delta'/3\}\varepsilon_{n}^{-d}\int_{\mathbb{R}^{d}}(s-s_{{\rm obs}})^{k_{2}}f_{n}(s\mid\theta)\,ds.\label{eq:Lemma3_1}
\end{align}
By Condition \ref{cond:kernel_prop}(ii), Condition \ref{cond:likelihood_moments}
and following the arguments in the proof of Lemma 3 of \cite{Li/Fearnhead:2018},
the right hand side of \eqref{eq:Lemma3_1} is $O_{p}(e^{-a_{n,\varepsilon}^{\alpha_{\delta}}c_{\delta}})$,
which is sufficient for $\pi_{B_{\delta}^{c}}\{(\theta-\theta_{0})^{k_{1}}(s-s_{{\rm obs}})^{k_{2}}\}$
to be $O_{p}(e^{-a_{n,\varepsilon}^{\alpha_{\delta}}c_{\delta}})$.

For the integration over $B_{\delta}$, by Lemma \ref{lem:LF_results}
(ii), 
\begin{eqnarray*}
\lefteqn{  \frac{\pi_{B_{\delta}}\{(\theta-\theta_{0})^{k_{1}}(s-s_{{\rm obs}})^{k_{2}}\}}{\pi_{B_{\delta}}(1)}
=  a_{n,\varepsilon}^{-k_{1}}\varepsilon_{n}^{k_{2}}\left\{ \frac{\tpi_{B_{\delta},tv}(t^{k_{1}}v^{k_{2}})}{\tpi_{B_{\delta},tv}(1)}+ \right.} & & \\
& & \left. \alpha_{n}^{-1}\frac{\int_{t(B_{\delta})}\int t^{k_{1}}v^{k_{2}}\pi(\theta_{0}+a_{n,\varepsilon}^{-1}t)r_{n}(s_{{\rm obs}}+\varepsilon_{n}v\mid\theta_{0}+a_{n,\varepsilon}^{-1}t)K(v)\,dvdt}{\tpi_{B_{\delta},tv}(1)}\right\} 
\{1+O_{p}(\alpha_{n}^{-1})\}
\end{eqnarray*}
where $r_{n}(s\mid\theta)$ is the scaled remainder $\alpha_{n}\{f_{n}(s\mid\theta)-\ftil_{n}(s\mid\theta)\}$.
In the above, the second term in the first brackets is $O_{p}(\alpha_{n}^{-1})$
by the proof of Lemma 6 of \cite{Li/Fearnhead:2018}. Then
\begin{align*}
\frac{\pi_{B_{\delta}}\{(\theta-\theta_{0})^{k_{1}}(s-s_{{\rm obs}})^{k_{2}}\}}{\pi_{B_{\delta}}(1)} & =a_{n,\varepsilon}^{-k_{1}}\varepsilon_{n}^{k_{2}}\left\{ \frac{\tpi_{B_{\delta},tv}(t^{k_{1}}v^{k_{2}})}{\tpi_{B_{\delta},tv}(1)}+O_{p}(\alpha_{n}^{-1})\right\} ,
\end{align*}
and the moments $\tpi_{B_{\delta},tv}(t^{k_{1}}v^{k_{2}})/\tpi_{B_{\delta},tv}(1)$
need to be evaluated. Theorem 1 of \cite{Li/Fearnhead:2018} gives the value of $\tpi_{B_{\delta},tv}(t)/\tpi_{B_{\delta},tv}(1)$,
and this is obtained by substituting the leading term of $\tpi_{\varepsilon,tv}(t,v)$, that is
$\pi(\theta_{0})g_{n}(t,v)$ as stated in Lemma \ref{lem:ABC_likelihood_expansion}, into the integrands. 
The other moments can be evaluated similarly, and give
\begin{align}
\frac{\tpi_{B_{\delta},tv}(t^{k_{1}}v^{k_{2}})}{\tpi_{B_{\delta},tv}(1)} & =\begin{cases}
b_{n}^{-1}\beta_{0}\{A(\theta_{0})^{1/2}W_{{\rm obs}}+a_{n}\varepsilon_{n}E_{G_{n}}(v)\}, & (k_{1},k_{2})=(1,0),\\
b_{n}^{-1}\beta_{0}\{A(\theta_{0})^{1/2}W_{{\rm obs}}E_{G_{n}}(v]+a_{n}\varepsilon_{n}E_{G_{n}}(vv^{T})\}, & (k_{1},k_{2})=(1,1),\\
E_{G_{n}}(v), & (k_{1},k_{2})=(0,1),\\
E_{G_{n}}(vv^{T}), & (k_{1},k_{2})=(0,2),
\end{cases}\nonumber \\
 & +O_{p}(a_{n,\varepsilon}^{-1})+O_{p}(a_{n}^{2}\varepsilon_{n}^{4}),\label{eq:tv_moments_expand}
\end{align}
where $b_{n}=1$ when $c_{\varepsilon}<\infty$, and $a_n\varepsilon_n$ when $c_{\varepsilon}=\infty$. By Lemma \ref{lem:ABC_likelihood_expansion}, $E_{G_{n}}(vv^{T})=\Theta_{p}(1)$.
Since $\alpha_{n}^{-1}=o(a_{n}^{-2/5})$, 
$\mbox{cov}_{\varepsilon}(\theta,s)=$$\varepsilon_{n}^{2}\beta_{0}\mbox{var}_{G_{n}}(v)+o_{p}(a_{n}^{-2/5}\varepsilon_{n}^2)$
and $\mbox{var}_{\varepsilon}(s)=\varepsilon_{n}^{2}\mbox{var}_{G_{n}}(v)\{1+o_{p}(a_{n}^{-2/5})\}$.
Thus
\begin{align}
\beta_{\varepsilon} & =\beta_{0}+o_{p}(a_{n}^{-2/5}),\label{eq:reg_coef_orders}
\end{align}
 and the lemma holds. \end{proof} 


For $A\subset\mathbbm{R}^{p}$ and $B\subset\mathbb{R}^{d}$, let
$\pi(A,B)=\int_{A}\int_{B}\pi(\theta)f_{n}(s\mid\theta)K\{\varepsilon_{n}^{-1}(s-s_{{\rm obs}})\}\varepsilon_{n}^{-d}\,dsd\theta$
and $\tpi(A,B)=\int_{A}\int_{B}\pi(\theta)\ftil_{n}(s\mid\theta)K\{\varepsilon_{n}^{-1}(s-s_{{\rm obs}})\}\varepsilon_{n}^{-d}\,dsd\theta$.
Denote the marginal mean values of $s$ for $\pi_{\varepsilon}(\theta,s\mid s_{{\rm obs}})$
and $\tpi_{\varepsilon}(\theta,s\mid s_{{\rm obs}})$ by $s_{\varepsilon}$
and $\ts_{\varepsilon}$ respectively.

\begin{proof}[of Lemma \ref{lem:adjABC_probability_expand}]

For (a), write $\Pi_{\varepsilon}(\theta^{*}\in B_{\delta}^{c}\mid s_{{\rm obs}})$
as $\pi[\mathbb{R}^{p},\{s:\theta^{*}(\theta,s)\in B_{\delta}^{c}\}]/\pi(\mathbb{R}^{p},\mathbb{R}^{d})$.
By Lemma \ref{lem:LF_results}, $\pi(\mathbb{R}^{p},\mathbb{R}^{d})=\pi_{\mathcal{P}}(1)=\Theta_{p}(a_{n,\varepsilon}^{d-p})$.
By the triangle inequality, 
\begin{align} \label{eq:tri}
\pi[\mathbb{R}^{p},\{s:\theta^{*}(\theta,s)\in B_{\delta}^{c}\}] & \leq\pi(B_{\delta/2}^{c},\mathbb{R}^{d})+\pi[B_{\delta/2},\{s:\|\beta_{\varepsilon}(s-s_{{\rm obs}})\|\geq\delta/2\}],
\end{align}
and it is sufficient that the right hand side of the above inequality
is $o_{p}(1)$. Since its first term is $\pi_{B_{\delta/2}^{c}}(1)$,
by Lemma \ref{lem:LF_results} the first term is $o_{p}(1)$.

When $\varepsilon_{n}=\Omega(a_{n}^{-7/5})$ or $\Theta(a_{n}^{-7/5})$,
by \eqref{eq:reg_coef_orders}, $\beta_{\varepsilon}-\beta_{0}=o_{p}(1)$
and so $\beta_{\varepsilon}$ is bounded in probability. For
any constant $\beta_{{\rm sup}}>0$ and $\beta\in\mathbb{R}^{p\times d}$
satisfying $\beta\leq\beta_{{\rm sup}}$,
\begin{align*}
\pi[B_{\delta/2},\{s:\|\beta(s-s_{{\rm obs}})\|\geq\delta/2\}] & \leq K\left(\varepsilon^{-1}\frac{\delta}{2\beta_{{\rm sup}}}\right)\varepsilon_{n}^{-d},
\end{align*}
and by Condition \ref{cond:kernel_prop}(iv), the second term in (\ref{eq:tri}) is $o_{p}(1)$.

When $\varepsilon_{n}=o(a_{n}^{-7/5})$, $\beta_{\varepsilon}$ is
unbounded and the above argument does not apply. Let $\delta_{1}$
be a constant less than $\delta_{0}$ such that $\inf_{\theta\in B_{\delta_{1}/2}}\lambda_{{\rm min}}\{A(\theta)^{-1/2}\}\geq m$
and $\inf_{\theta\in B_{\delta_{1}/2}}\lambda_{{\rm min}}\{Ds(\theta)\}\geq m$
for some positive constant $m$. In this case, it is sufficient to
consider $\delta<\delta_{1}$. By Condition \ref{cond:sum_approx},
\begin{align*}
 & r_{n}(s\mid \theta)\leq a_{n}^{d}|A(\theta)|^{1/2}r_{{\rm max}}[a_{n}A(\theta)^{-1/2}\{s-s(\theta)\}].
\end{align*}
Using the transformation $t=t(\theta)$ and $v=v(s)$, 
$f_{n}(s\mid \theta)=\ftil_{n}(s\mid \theta)+\alpha_{n}^{-1}r_{n}(s\mid \theta)$
and applying the Taylor expansion of $s(\theta_{0}+xt)$ around $x=0$,
\begin{align*}
 & \pi[B_{\delta/2},\{s:\|\beta_{\varepsilon}(s-s_{{\rm obs}})\|\geq\delta/2\}] \leq \\
 & c\int_{t(B_{\delta/2})}\int_{\|\beta_{\varepsilon}\varepsilon_{n}v\|\geq\delta/2}N[A(\theta_{0}+a_{n}^{-1}t)^{-1/2}\{Ds(\theta_{0}+e_{n}^{(1)}t)t-A(\theta_{0})^{1/2}W_{{\rm obs}}-a_{n}\varepsilon_{n}v\};0,I_{d}]K(v)\,dvdt\\
 & +c\int_{t(B_{\delta/2})}\int_{\|\beta_{\varepsilon}\varepsilon_{n}v\|\geq\delta/2}r_{{\rm max}}[A(\theta_{0}+a_{n}^{-1}t)^{-1/2}\{Ds(\theta_{0}+e_{n}^{(1)}t)t-A(\theta_{0})^{1/2}W_{{\rm obs}}-a_{n}\varepsilon_{n}v\}]K(v)\,dvdt,
\end{align*}
for some positive constant $c$. To show that the right hand side
of the above inequality is $o_{p}(1)$, consider a function $g_{4}(\cdot)$
in $\mathbb{R}^{d}$ satisfying that $g_{4}(v)$ can be written as
$\overline{g}_{4}(\|v\|)$ and $\overline{g}_{4}(\cdot)$ is decreasing.
Let $A_{n}(t)=A(\theta_{0}+a_{n}^{-1}t)^{-1/2}$, $C_{n}(t)=Ds(\theta_{0}+\xi_{1})$
and $c=A(\theta_{0})^{1/2}W_{{\rm obs}}$. For each $n$ divide $\mathbb{R}^{p}$
into $V_{n}=\{t:\|C_{n}(t)t\|/2\geq\|c+a_{n}\varepsilon_{n}v\|\}$
and $V_{n}^{c}$. In $V_{n}$, $\|A_{n}(t)\{C_{n}(t)t-c-a_{n}\varepsilon_{n}v\}\|\geq m^{2}\|t\|/2$
and in $V^{c}$, $\|t\|\leq2m^{-1}\|c+a_{n}\varepsilon_{n}v\|$. Then
\begin{align*}
 & \int_{t(B_{\delta/2})}\int_{\|\beta_{\varepsilon}\varepsilon_{n}v\|\geq\delta/2}g_{4}[A_{n}(t)\{C_{n}(t)t-c-a_{n}\varepsilon_{n}v\}]K(v)\,dvdt\\
\leq & \int_{\|\beta_{\varepsilon}\varepsilon_{n}v\|\geq\delta/2}\left\{ \int_{\mathbb{R}^{p}}\overline{g}_{4}(m^{2}\|t\|/2)\,dt+\sup_{v\in\mathbb{R}^{p}}g_{4}(v)\int_{_{V_{n}^{c}}}1\,dt\right\} K(v)\,dv,
\end{align*}
where $\int_{_{V_{n}^{c}}}1\,dt$ is the volume of $V_{n}^{c}$ in
$\mathbb{R}^{p}$. Then since $\beta_{\varepsilon}\varepsilon_{n}=o_{p}(1)$,
$a_{n}\varepsilon_{n}=o_{p}(1)$ and $\int_{V_{n}^{c}}1\,dt$ is proportional
to $\|c+a_{n}\varepsilon_{n}v\|^{p}$, the right hand side of the
above inequality is $o_{p}(1)$. This implies $\pi(B_{\delta/2},\{s:\|\beta_{\varepsilon}(s-s_{{\rm obs}})\|\geq\delta/2\})=o_{p}(1)$.

Therefore in both cases $\Pi_{\varepsilon}(\theta^{*}\in B_{\delta}^{c}\mid s_{{\rm obs}})=o_{p}(1)$.
For $\tPi_{\varepsilon}(\theta^{*}\in B_{\delta}^{c}\mid s_{{\rm obs}})$,
since the support of its prior is $B_{\delta}$, there is no probability
mass outside $B_{\delta}$ , i.e. $\tPi_{\varepsilon}(\theta^{*}\in B_{\delta}^{c}\mid s_{{\rm obs}})=0$.
Therefore (a) holds.

For (b), 
\begin{align*}
 & \mbox{{\rm sup}}_{A\in\mathscr{B}^{p}}\left|\Pi_{\varepsilon}(\theta^{*}\in A_{\theta}\cap B_{\delta}\mid s_{{\rm obs}})-\tPi_{\varepsilon}(\theta^{*}\in A_{\theta}\cap B_{\delta}\mid s_{{\rm obs}})\right|\\
= & \frac{\mbox{{\rm sup}}_{A\in\mathscr{B}^{p}}\left|\pi(\mathbb{R}^{p},\{s:\theta^{*}(\theta,s)\in A_{\theta}\cap B_{\delta}\})-\tpi(\mathbb{R}^{p},\{s:\theta^{*}(\theta,s)\in A_{\theta}\cap B_{\delta}\})\right|}{\tpi_{B_{\delta}}(1)}+o_{p}(1)\\
\leq & \alpha_{n}^{-1}\frac{\int_{B_{\delta}}\int_{\mathbb{R}^{d}}\pi(\theta)|r_{n}(s\mid \theta)|K\{\varepsilon_{n}^{-1}(s-s_{{\rm obs}})\}\varepsilon_{n}^{-d}\,dsd\theta}{\tpi_{B_{\delta}}(1)}+o_{p}(1).
\end{align*}
Then by the proof of Lemma 6 of \cite{Li/Fearnhead:2018}, (b) holds.

For (c), to begin with, $a_{n}(\theta_{\varepsilon}^{*}-\ttheta_{\varepsilon}^{*})=a_{n}(\theta_{\varepsilon}-\ttheta_{\varepsilon})-a_{n}\beta_{\varepsilon}(s_{\varepsilon}-\ts_{\varepsilon})$.
By Lemma \ref{lem:LF_results}, $a_{n}(\theta_{\varepsilon}-\ttheta_{\varepsilon})=o_{p}(1)$.
For $a_{n}\beta_{\varepsilon}(s_{\varepsilon}-\ts_{\varepsilon})$,
similar to the arguments of the proof of Lemma \ref{lem:reg_coef},
\begin{align*}
s_{\varepsilon}-s_{{\rm obs}} & =\varepsilon_{n}\left\{ \frac{\tpi_{B_{\delta,tv}}(v)}{\tpi_{B_{\delta,tv}}(1)}+O_{p}(\alpha_{n}^{-1})\right\} \{1+O_{p}(\alpha_{n}^{-1})\},\quad\ts_{\varepsilon}-s_{{\rm obs}}=\varepsilon_{n}\frac{\tpi_{B_{\delta,tv}}(v)}{\tpi_{B_{\delta,tv}}(1)}\{1+O_{p}(\alpha_{n}^{-1})\}.
\end{align*}
Then $a_{n}\beta_{\varepsilon}(s_{\varepsilon}-\ts_{\varepsilon})=O_{p}(\alpha_{n}^{-1}a_{n}\varepsilon_{n})$
which is $o_{p}(1)$ if $\varepsilon_{n}=o(a_{n}^{-3/5})$. Therefore
the first part of (c) holds. Since $\ttheta_{\varepsilon}^{*}=\ttheta_{\varepsilon}-\beta_{\varepsilon}(\ts_{\varepsilon}-s_{{\rm obs}})$,
by the expansion of $\ttheta_{\varepsilon}$ in Lemma \ref{lem:ABC_probability_expand}(c),
the above expansion of $\ts_{\varepsilon}-s_{{\rm obs}}$ and \eqref{eq:tv_moments_expand},
the second part of (c) holds.
\end{proof}

\subsection*{Proof for Results in Section \ref{ABCacc_rate}}

\begin{proof}[of Theorem \ref{thm:acceptance_rate}]

The integrand of $p_{{\rm acc},q}$ is similar to that of $\pi_{\mathbb{R}^{p}}(1)$.
The expansion of $\pi_{\mathbb{R}^{p}}(1)$ is given in Lemma \ref{lem:LF_results}(ii),
and following the same reasoning, $p_{{\rm acc},q}$ can be expanded as
$\varepsilon_{n}^{d}\int_{B_{\delta}}\int_{\mathbb{R}^{d}}q_{n}(\theta)\ftil(s_{{\rm obs}}+\varepsilon_{n}v\mid \theta)K(v)\,dvd\theta\{1+o_{p}(1)\}$.
With transformation $t=t(\theta)$, plugging the expression of $q_{n}(\theta)$
and $\tpi_{\varepsilon,tv}(t,v)$ gives that 
\begin{align*}
p_{{\rm acc},q} & =(a_{n,\varepsilon}\varepsilon_{n})^{d}\int_{t(B_{\delta})}(r_{n,\varepsilon})^{-p}q(r_{n,\varepsilon}^{-1}t-c_{\mu})\frac{\tpi_{\varepsilon,tv}(t,v)}{\pi_{\delta}(\theta_{0}+a_{n,\varepsilon}^{-1}t)}\,dvdt\{1+o_{p}(1)\},
\end{align*}
where $r_{n,\varepsilon}=\sigma_{n}/a_{n,\varepsilon}^{-1}$ and $c_{\mu,n}=\sigma_{n}(\mu_{n}-\theta_{0})$.
By the assumption of $\mu_{n}$, denote the limit of $c_{\mu,n}$
by $c_{\mu}$. Then by Lemma \ref{lem:ABC_likelihood_expansion},
$p_{{\rm acc},q}$ can be expanded as 
\begin{align}
p_{{\rm acc},q} & =(a_{n,\varepsilon}\varepsilon_{n})^{d}\int_{t(B_{\delta})\times\mathbb{R}^{d}}(r_{n,\varepsilon})^{-p}q(r_{n,\varepsilon}^{-1}t-c_{\mu,n})g_{n}(t,v)\,dvdt\{1+o_{p}(1)\}.\label{eq:accrate_expand}
\end{align}
Denote the leading term of the above by $Q_{n,\varepsilon}$. 

For (1), when $c_{\varepsilon}=0$, since $\sup_{t\in\mathbb{R}^{p}}g_{n}(t,v)\leq c_{1}K(v)$
for some positive constant $c_{1}$, $Q_{n,\varepsilon}$ is upper
bounded by $(a_{n}\varepsilon_{n})^{d}c_{1}$ almost surely. Therefore
$p_{{\rm acc},q}\rightarrow0$ almost surely as $n\rightarrow\infty$. When $r_{n,\varepsilon}\rightarrow\infty$,
since $q(\cdot)$ is bounded in $\mathbb{R}^{p}$ by some positive
constant $c_{2}$, $Q_{n,\varepsilon}$ is upper bounded by $(r_{n,\varepsilon})^{-p}c_{2}(a_{n,\varepsilon}\varepsilon_{n})^{d}\int_{\mathbb{R}^{p}\times\mathbb{R}^{d}}g_{n}(t,v)\,dvdt$.
Therefore $p_{{\rm acc},q}\rightarrow0$ in probability as $n\rightarrow\infty$ since
$\int_{\mathbb{R}^{p}\times\mathbb{R}^{d}}g_{n}(t,v)\,dvdt=\Theta_{p}(1)$
by Lemma \ref{lem:ABC_likelihood_expansion}. 

For (2), let $\tilde{t}(\theta)=r_{n,\varepsilon}^{-1}t(\theta)-c_{\mu,n}$
and $\tilde{t}(A)$ be the set $\{\phi:\phi=\tilde{t}(\theta)\text{ for some }\theta\in A\}$.
Since $\tilde{t}=\sigma_{n}^{-1}(\theta-\theta_{0})-c_{\mu,n}$ and
$\sigma_{n}^{-1}\rightarrow\infty$, $\tilde{t}(B_{\delta})$ converges
to $\mathbb{R}^{p}$ in probability as $n\rightarrow\infty$. With the transformation $\tilde{t}=\tilde{t}(\theta)$,
\[
Q_{n,\varepsilon}=\begin{cases}
(a_{n}\varepsilon_{n})^{d}\int_{\tilde{t}(B_{\delta})\times\mathbb{R}^{d}}q(\tilde{t})g_{n}\{r_{n,\varepsilon}(\tilde{t}+c_{\mu,n}),v\}\,d\tilde{t}dv, & \ c_{\varepsilon}<\infty,\\
\int_{\tilde{t}(B_{\delta})\times\mathbb{R}^{d}}q(\tilde{t})g_{n}'\{r_{n,\varepsilon}(\tilde{t}+c_{\mu,n}),v'\}\,d\tilde{t}dv', & \ c_{\varepsilon}=\infty.
\end{cases}
\]
By Lemma \ref{lem:continuity_of_integrals} and the continuous mapping
theorem,
\[
Q_{n,\varepsilon}\rightarrow\begin{cases}
c_{\varepsilon}^{d}\int_{\mathbb{R}^{p}\times\mathbb{R}^{d}}q(\tilde{t})g\{r_{1}(\tilde{t}+c_{\mu}),v\}\,d\tilde{t}dv, & \ c_{\varepsilon}<\infty,\\
\int_{\mathbb{R}^{p}\times\mathbb{R}^{d}}q(\tilde{t})g\{r_{1}(\tilde{t}+c_{\mu}),v\}\,d\tilde{t}dv, & \ c_{\varepsilon}=\infty,
\end{cases}
\]
in distribution as $n\rightarrow\infty$. Since the limits above are $\Theta_{p}(1)$, $p_{{\rm acc},q}=\Theta_{p}(1)$. 

For (3), when $c_{\varepsilon}=\infty$ and $r_{1}=0$, in the above,
the limit of $Q_{n,\varepsilon}$ in distribution is $\int_{\mathbb{R}^{p}\times\mathbb{R}^{d}}q(\tilde{t})g(0,v)\,d\tilde{t}dv=1$.
Therefore $p_{{\rm acc},q}$ converges to $1$ in probability as $n\rightarrow\infty$. 
\end{proof}

\bibliographystyle{biometrika}
\bibliography{reference}

\end{document}